\documentclass{article}
\usepackage[a4paper,margin=1in]{geometry}

\usepackage[applemac]{inputenc} 		
\usepackage[T1]{fontenc}    		
\usepackage[colorlinks]{hyperref}      
\usepackage{url}            			
\usepackage{amsthm}
\usepackage{booktabs}       		
\usepackage{amsfonts}       		
\usepackage{amsmath}

\usepackage{nicefrac}       		
\usepackage{microtype}      		
\usepackage{lipsum}				
\usepackage[square,numbers]{natbib}
\usepackage{mathtools}
\usepackage{algorithm}
\usepackage{algorithmicx}
\usepackage[algo2e,linesnumbered,vlined,ruled]{algorithm2e}
\usepackage{algpseudocode}
\usepackage{graphicx}
\usepackage{indentfirst,latexsym,bm}
\usepackage{amsmath}
\usepackage{subfigure}
\usepackage{amssymb}
\usepackage{xcolor}
\usepackage{comment}
\usepackage{enumitem}
\usepackage{bbm}
\usepackage{tikz}
\usepackage{mdframed}
\usepackage{nicematrix}
\usepackage{bbding}
\usepackage{pifont}

\newcommand{\xmark}{\color{red}\ding{55}}
\definecolor{lavender}{rgb}{0.9, 0.9, 0.98}

\usetikzlibrary{arrows.meta,positioning}
\usepackage{pgfplots}
\pgfplotsset{compat=newest}
\usepgfplotslibrary{fillbetween}
\usetikzlibrary{shapes,decorations}
\usetikzlibrary{fit}

\newcommand{\lya}{\mathcal{R}}

\newcommand{\R}{\mathbb{R}}
\newcommand{\N}{\mathbb{N}}
\newcommand{\Rn}{\mathbb{R}^d}

\newcommand{\Prob}{\mathbb{P}}
\newcommand{\Exp}{\mathbb{E}}

\newcommand{\dist}{\mathrm{dist}}
\newcommand{\crit}{\mathrm{crit}}

\newcommand{\cG}{\mathcal{G}}

\newcommand{\sL}{L}
\newcommand{\sG}{G}
\newcommand{\sH}{H}

\newcommand{\sC}{C}
\newcommand{\sD}{D}

\newcommand{\sA}{A}
\newcommand{\sB}{B}

\newcommand{\nes}{\hat y}

\newcommand{\cO}{\mathcal O}
\newcommand{\cA}{\mathcal A}

\newcommand{\cF}{\mathcal F}

\newcommand{\iprod}[2]{\langle #1, #2 \rangle}

\newcommand{\SGD}{\mathrm{SGD}}
\newcommand{\SGDM}{\mathrm{SGDM}}
\newcommand{\SMG}{\mathrm{SMG}}

\newcommand{\RR}{\mathrm{RR}}
\newcommand{\RRM}{\mathrm{RRM}}
\newcommand{\RRHB}{\mathrm{RRHB}}
\newcommand{\RRNAG}{\mathrm{RRNAG}}

\newcommand{\IGM}{\mathrm{IGM}}
\newcommand{\SOM}{\mathrm{SOM}}

\definecolor{purp}{RGB}{152,24,147}
\definecolor{bluep}{RGB}{0,128,255}
\definecolor{redp}{RGB}{255,0,0}

\definecolor{OxfordBlue}{rgb}{0,0.106,0.329}   
\definecolor{UMRed}{rgb}{0.73,0.09,0.19}   
\definecolor{CUBrown}{RGB}{152,95,42}   
\definecolor{LightBrown}{RGB}{219,199,181}   
\definecolor{MyBlue}{RGB}{123,144,210} 
\definecolor{LightBlue}{RGB}{229,232,247}   
\definecolor{MyGreen}{RGB}{165,222,228}   
\definecolor{LightGreen}{RGB}{228,245,247} 
\definecolor{turql}{RGB}{53,130,134}

\definecolor{MyGrayTwo}{RGB}{180,180,180} 
\definecolor{McBlack}{RGB}{39,37,31} 
\definecolor{McRed}{RGB}{218,41,28} 
\definecolor{McYellow}{RGB}{255,199,44} 
\definecolor{DeepBlue}{RGB}{2,103,178}
\definecolor{DeepRed}{RGB}{192,0,0}
\definecolor{MyOrange}{RGB}{237,156,89}

\definecolor{MyGray}{RGB}{180,180,180}    
\definecolor{MyRed}{RGB}{192,0,0}  

\usepackage{hyperref}[6.83]

\hypersetup{
    colorlinks=true, 
    linkcolor=blue,  
    citecolor=blue, 
    urlcolor=black 
}

\RequirePackage[capitalize,nameinlink]{cleveref}
\crefname{appendix}{Appendix}{Appendices}
\Crefname{appendix}{Appendix}{Appendices}

\crefformat{equation}{\textup{#2(#1)#3}}
\crefrangeformat{equation}{\textup{#3(#1)#4--#5(#2)#6}}
\crefmultiformat{equation}{\textup{#2(#1)#3}}{ and \textup{#2(#1)#3}}
{, \textup{#2(#1)#3}}{, and \textup{#2(#1)#3}}
\crefrangemultiformat{equation}{\textup{#3(#1)#4--#5(#2)#6}}%
{ and \textup{#3(#1)#4--#5(#2)#6}}{, \textup{#3(#1)#4--#5(#2)#6}}{, and \textup{#3(#1)#4--#5(#2)#6}}

\Crefformat{equation}{#2Equation~\textup{(#1)}#3}
\Crefrangeformat{equation}{Equations~\textup{#3(#1)#4--#5(#2)#6}}
\Crefmultiformat{equation}{Equations~\textup{#2(#1)#3}}{ and \textup{#2(#1)#3}}
{, \textup{#2(#1)#3}}{, and \textup{#2(#1)#3}}
\Crefrangemultiformat{equation}{Equations~\textup{#3(#1)#4--#5(#2)#6}}%
{ and \textup{#3(#1)#4--#5(#2)#6}}{, \textup{#3(#1)#4--#5(#2)#6}}{, and \textup{#3(#1)#4--#5(#2)#6}}

\crefdefaultlabelformat{#2\textup{#1}#3}

\theoremstyle{plain}
\newtheorem{theorem}{Theorem}[section]
\newtheorem{lemma}[theorem]{Lemma}
\newtheorem{corollary}[theorem]{Corollary}
\newtheorem{proposition}[theorem]{Proposition}
\newtheorem{remark}[theorem]{Remark}
\newtheorem{assumption}[theorem]{Assumption}

\crefformat{assumption}{#2Assumption~#1#3}
\Crefformat{assumption}{#2Assumption~#1#3}

\theoremstyle{plain}
\newtheorem{definition}[theorem]{Definition}

\title{Random Reshuffling with Momentum: Complexity Bounds and Last-iterate Convergence}

\author{
Junwen Qiu\thanks{Industrial Systems Engineering and Management, National University of Singapore}
\and
Bohao Ma\thanks{School of Data Science, The Chinese University of Hong Kong, Shenzhen, Shenzhen, Guangdong, China\\
Email: \href{mailto:jwqiu@nus.edu.sg}{\texttt{jwqiu@nus.edu.sg}}, \href{mailto:bohaoma@link.cuhk.edu.cn}{\texttt{bohaoma@link.cuhk.edu.cn}}, 
\href{mailto:andremilzarek@cuhk.edu.cn}{\texttt{andremilzarek@cuhk.edu.cn}}}
\and
Andre Milzarek\footnotemark[2]
}
\begin{document}
\maketitle

\begin{abstract}
Random reshuffling with momentum ($\RRM$) corresponds to the $\SGD$ optimizer with the \emph{momentum} option enabled, as found in many machine learning libraries such as PyTorch and TensorFlow. Despite its widespread use, the convergence properties of $\RRM$ do not seem to be well understood.
This work establishes new complexity bounds and asymptotic convergence guarantees for popular versions of $\RRM$ using stochastic heavy-ball momentum, Nesterov acceleration, and mini-batches in a general nonconvex setting. In particular, we prove that the base variant of $\RRM$ achieves the complexity $\mathcal O(n^{-1/3}((1-\beta^n)T)^{-2/3})$, where $n$ denotes the number of samples, $\beta \in [0,1)$ is a momentum parameter, and $T$ is the total number of epochs. 
On the asymptotic side, we show that every accumulation point of the iterates $\{x^k\}_k$ generated by $\RRM$ is a stationary point of the problem. For definable objectives---a broad and common class of functions including, e.g., semialgebraic, globally subanalytic, and log-exp functions---we strengthen this subsequential result to last-iterate convergence to a single stationary point. Moreover, improved asymptotic complexity bounds are presented that are based on the additional geometric properties of definable functions.
\end{abstract}

\section{Introduction}
\label{intro}
Many minimization and learning tasks can be formulated as finite-sum problems \cite{bottou2018optimization}, where the goal is to minimize the sum of a potentially large number of component functions:
%
\begin{equation}
	\label{SO} \min_{x\in \Rn}~f(x) := \frac{1}{n}\;{\sum}_{i=1}^{n}\,f_i(x).
\end{equation}
Here, each function $f_i:\Rn\to\R$, $i\in[n]:=\{1,\dots,n\}$, is assumed to be continuously differentiable. Stochastic gradient methods leverage stochastic approximation techniques and are popular approaches for this class of finite-sum problems, \cite{robbins1951stochastic,chung1954stochastic,tseng1998incremental,ghadimi2012optimal,gurbu2019}. 
Recent research has shown that the use of \emph{without-replacement sampling schemes} in stochastic algorithms can have a positive impact on convergence and performance \cite{bottou2009curiously,bottou2012,gurbu2019,mishchenko2020random}. Such sampling strategies are widely applied in practice and are commonly referred to as random reshuffling ($\RR$). In fact, the basic stochastic gradient descent method ($\SGD$) is implemented using random reshuffling when solving large-scale machine learning problems; see \cite{bottou2012,tseng1998incremental} and the documentation in Scikit-learn\footnote{\url{https://scikit-learn.org/stable/modules/sgd.html}} and TensorFlow\footref{tf}, \cite{tensorflow2015}. A popular modification of $\RR$ is to incorporate \emph{momentum} in each iteration \cite{polyak1964some,nesterov1983method,werbos1974beyond,rumelhart1986learning,sutskever2013importance}. In the following, we present the core loop of random reshuffling with momentum ($\RRM$). A more formal introduction of the algorithm (with mini-batches) can be found in \Cref{sec:convergence analysis}. Let $\pi^k = (\pi_1^k,\dots,\pi_n^k)$ denote a random permutation of the numbers $\{1,\dots,n\}$ in the $k$-th epoch. Then, the update of $\RRM$ is given by: 
\begin{equation*}
    \begin{aligned} \text{$k$-th outer loop\,/} & \\ \text{\,$k$-th epoch} \end{aligned}\;\;\left[\;  
    \begin{aligned}
    & \text{Set} \; y_0 = \tilde x^k\;\; \text{and}\;\;  y_1 = x^k\\
    &\textbf{for} \; i=1,\ldots,n \; \textbf{do} \\
    &\qquad y_{i+1} = y_i - \alpha_k \nabla f_{\pi^k_i}(y_i + \lambda(y_i - y_{i-1})) + \beta(y_i-y_{i-1})\\[1mm]
& \text{Set} \; \tilde x^{k+1} = y_n \;\; \text{and}\;\;  x^{k+1} = y_{n+1}.
\end{aligned}
\right.
\end{equation*}
A key motivation for this work stems from the observation that $\RRM$ is a standard optimizer included in prominent machine learning libraries such as TensorFlow \cite{tensorflow2015} and PyTorch \cite{paszke2019pytorch}. For example, the code snippet shown below from TensorFlow\footnote{\url{https://www.tensorflow.org/api_docs/python/tf/keras/optimizers/SGD}\label{tf}} implements $\RRM$ with the constant learning rate $\alpha_k=0.1$ and momentum parameters $\beta=0.9$, $\lambda=0$: 
\begin{equation} \label{eq:code}
    \texttt{tf.keras.optimizers.SGD(learning\_rate = 0.1, momentum = 0.9)}.
\end{equation}
This specific variant of $\RRM$ is also known as the random reshuffling heavy-ball method ($\RRHB$). By selecting \texttt{`nesterov = True'} in \eqref{eq:code}, TensorFlow calls $\RRM$ with $\alpha_k=0.1$ and $\lambda = \beta =0.9$, which refers to random reshuffling with Nesterov acceleration ($\RRNAG$).

\subsection{Related work}
Originally developed by Werbos \cite{werbos1974beyond} and Rumelhart, Hinton, and Williams \cite{rumelhart1986learning} for neural network training, $\RRM$ has found widespread application in solving finite-sum problems \cite{rumelhart1986learning,sutskever2013importance,tseng1998incremental,yu2019linear}. However, the convergence behavior of $\RRM$ is less understood compared to $\RR$ and $\SGDM$\footnote{$\SGDM$ refers to the stochastic momentum variant(s) of $\SGD$, where the stochastic gradient information is generated via \emph{with-replacement} sampling schemes. In contrast to $\RR$-type methods (and $\RRM$), such stochastic gradients are then unbiased estimators of the full gradient.}. The challenges arise from two main aspects. Firstly, unlike $\SGD$, the stochastic gradients in $\RRM$ follow a \emph{without-replacement} sampling strategy and are not unbiased estimators of the true gradient. Secondly, the momentum steps in $\RRM$ tend to accumulate the stochastic errors from previous iterations, thus complicating the overall analysis; we refer to \Cref{subsec:challenges} for more detailed discussions.

Despite these challenges, there has been some important progress in understanding the convergence properties of $\RRM$. In \cite{tseng1998incremental}, Tseng studied an incremental gradient method with heavy-ball momentum ($\IGM$) which can be seen as a deterministic analogue of $\RRHB$. Tseng showed that accumulation points of a subsequence of iterates generated by $\IGM$ (equipped with an adaptive descent mechanism) are stationary points if $\beta < {0.5}^{1/n}$. Tran et al. \cite[Theorem 3]{tran2021smg} derived the following complexity result for $\RRHB$ under a bounded gradient assumption:
\begin{equation} \label{eq:tran-complexity} {\min}_{k=1,\ldots,T} \; \Exp[\|\nabla f(x^k)\|^2] = \cO((1-\beta^n)^{-1}  [T^{-2/3} + \beta^n]), \end{equation}
if the same (\emph{shuffle-once}) permutation $\pi^k \equiv \pi$ is used for all $k$. This complexity bound requires the momentum parameter $\beta$ to be sufficiently small to achieve a desired accuracy and it does \emph{not} tend to $0$ if $T \to \infty$. To our knowledge, the bound \eqref{eq:tran-complexity} seems to be the only available complexity result for $\RRM$; other existing bounds only apply to modified variants of $\RRM$ \cite{tran2021smg,tran2022nesterov}. 
In \cite{tran2021smg}, Tran et al.\ proposed a shuffling momentum-type gradient method ($\SMG$) that achieves the improved complexity 
$\min_{k=1,\ldots,T} \; \Exp[\|\nabla f(x^k)\|^2]=\cO((1-\beta)^{-1}n^{-1/3}T^{-2/3})$. 
Except for the additional factor $(1-\beta)^{-1}$, this bound exactly matches the iteration complexity of $\RR$, cf.\ \cite{mishchenko2020random} and Table~\ref{table:super-nice}.
$\SMG$ is a modified version of $\RRHB$ that employs a fixed momentum term throughout each outer loop (epoch). This update differs from conventional momentum methods used in practice. Specifically, the momentum component in $\SMG$ operates at the epoch level rather than the (inner) iterate level, and it is computed as the averaged gradient of all component functions evaluated during the preceding epoch. In the convex setting, Tran et al. \cite{tran2022nesterov} established convergence guarantees for the function value sequence for a variant of $\RRNAG$ with analogous structural modifications. Very recently, Liang and Xu \cite{JMLR:v26:24-1243} have analyzed the last-iterate convergence properties of $\SMG$. Specifically, they establish last-iterate convergence, $x^k \to x^*$, of $\SMG$ under the assumption that the objective function is coercive and satisfies a variant of the Kurdyka-{\L}ojasiewicz inequality. In addition, explicit convergence rates are shown under the stronger {\L}ojasiewicz inequality. The work \cite{JMLR:v26:24-1243} extends the theoretical analysis of $\RR$ developed in \cite{li2023convergence} to the shuffling momentum-type method $\SMG$. 

Finally, in a recent study, Josz and Lai \cite{josz2023global} established stability guarantees for $\RRM$ in a general nonsmooth setting if the objective function is tame and coercive. Josz and Lai proved that the trajectories of $\RRM$ can be approximated by subgradient trajectories and the iterates will eventually remain within a neighborhood of the set of stationary points.

\begin{table}[t]
\centering
{\footnotesize
\setlength{\tabcolsep}{5pt}
\NiceMatrixOptions{cell-space-limits=1pt}
\begin{NiceTabular}{|c| p{2.3cm}p{2cm}p{2cm}p{2cm}p{2.3cm}|c|}%
 [ 
   code-before = 
     \rectanglecolor{lavender!30}{2-2}{10-2}
    \rectanglecolor{lavender!30}{2-4}{10-4}
        \rectanglecolor{lavender!30}{2-6}{10-6}
 ]
\toprule
\Block{3-1}{\textbf{Alg.}} & \Block[c]{1-2}{\textbf{Conditions}} & & \Block[c]{1-3}{\textbf{Convergence: Nonconvex Case}} & & &  \Block{3-1}{\textbf{Ref.}}  \\ \Hline
& \Block{2-1}{bounded grad. \emph{not required}} 
& \Block{2-1}{$\beta\in[0,1)$ is \emph{free}} & \Block{2-1}{iteration complexity}
& \Block{2-1}{global conv. $\nabla f(x^k)\to0$} & \Block{2-1}{iterate conv. $x^k \to x^*$} & \\
 & & & & & &  \\[1mm] \Hline
\Block{2-1}{$\RR$} & \Block{1-1}{\checkmark} & \Block{1-1}{--} & \Block{1-1}{\checkmark} & \Block{1-1}{\xmark}  & \Block{1-1}{\xmark} & \cite{mishchenko2020random}\\
& \Block{1-1}{\checkmark} & \Block{1-1}{--} & \Block{1-1}{\xmark} & \Block{1-1}{\checkmark} & \Block{1-1}{\checkmark}  & \cite{li2023convergence}\\ \Hline
\Block{2-1}{$\SMG^{\textcolor{blue}{*}}$} & \Block{1-1}{\checkmark} & \Block{1-1}{\checkmark} & \Block{1-1}{\checkmark} & \Block{1-1}{\xmark} & \Block{1-1}{\xmark} & \cite{tran2021smg}\\ 
& \Block{1-1}{\checkmark} & \Block{1-1}{\checkmark} & \Block{1-1}{\checkmark} & \Block{1-1}{\checkmark} & \Block{1-1}{\checkmark} & \cite{JMLR:v26:24-1243} \\ \Hline
\Block{3-1}{$\RRM$} 
 & \Block{1-1}{\xmark} & \Block{1-1}{\xmark} & \Block{1-1}{\; \checkmark\,${}^{\textcolor{blue}{\dagger}}$} & \Block{1-1}{\xmark} & \Block{1-1}{\xmark} & \cite{tran2021smg}\\ 
& \Block[borders={bottom,tikz=densely dotted}]{1-1}{\xmark} & \Block[borders={bottom,tikz=densely dotted}]{1-1}{\xmark} & \Block[borders={bottom,tikz=densely dotted}]{1-1}{\xmark} & \Block[borders={bottom,tikz=densely dotted}]{1-1}{\; \checkmark\,${}^{\textcolor{blue}{\ddagger}}$} & \Block[borders={bottom,tikz=densely dotted}]{1-1}{\xmark} & \Block[borders={bottom,tikz=densely dotted}]{1-1}{\cite{tseng1998incremental}}\\
 & \Block{1-1}{\checkmark} & \Block{1-1}{\checkmark} & \Block{1-1}{Thm. \ref{thm:complexity}} & \Block{1-1}{Thm. \ref{thm:global_convergence}} & \Block{1-1}{Thm. \ref{thm:iter_convergence}}& Ours\\ \Hline
\bottomrule
\end{NiceTabular}
}
\vspace{1ex}
\caption{Comparison of convergence guarantees for $\RR$ and its momentum variants for the case $\lambda = 0$.
\endgraf
\setlength{\parindent}{1ex} \setlength{\baselineskip}{11pt}
\noindent \hspace{1pt} ${}^{\textcolor{blue}{*}}$ {\footnotesize The momentum term of $\SMG$ (proposed in \cite{tran2021smg}) is fixed within an outer loop (epoch). This method is different from the stochastic momentum methods used in practice.}\\
\hspace*{1pt} ${}^{\textcolor{blue}{\dagger}}$ {\footnotesize To achieve arbitrary accuracies, this requires $\beta$ to be close to $0$ (see Table~\ref{table:super-nice} and Appendix \ref{app:table}).}\\
\hspace*{1pt} ${}^{\textcolor{blue}{\ddagger}}$ {\footnotesize In \cite[Proposition 3.4]{tseng1998incremental}, the gradients $\nabla f_i$ are assumed to be bounded on a certain level set and a line search-type strategy is used on a predefined subsequence $\{\ell_k\}_k$ of iterates; based on  $\beta^n < \frac12$; convergence ``$\|\nabla f(x^k)\| \to 0$'' only holds on the subsequence $\{\ell_k\}_k$.} 
 \endgraf}
\vspace{-3mm}
\label{table:contribution}
\end{table}

\subsection{Contributions} Motivated by the practical relevance of random reshuffling with momentum, this work aims to address some of the current theoretical gaps and provides an in-depth analysis of $\RRM$ in the nonconvex setting under mild assumptions. Our key contributions are outlined below. An additional overview and comparison is given in \Cref{table:contribution}. \\[1mm]
\noindent\emph{Iteration Complexity.} Under the assumption that each $f_i$ is Lipschitz smooth and bounded from below, we establish the bounds $\min_{k=1,\ldots,T} \|\nabla f(x^k)\|^2=\cO((1-\beta^m)^{-1}T^{-2/3})$ and
\begin{equation} \label{eq:wuhu} {\min}_{k=1,\ldots,T} \; \Exp[\|\nabla f(x^k)\|^2] = \cO((1-\beta^m)^{-1}n^{-1/3}T^{-2/3}), \quad m = n/b, \end{equation} 
for $\RRM$. 
Here, $b$ denotes the mini-batch size and the in-expectation result \eqref{eq:wuhu} is obtained when a uniform sampling scheme is used. The derived complexity bounds are valid for arbitrary choices of the momentum parameters $\beta\in[0,1)$ and $\lambda \in [0, \frac{\beta}{1-\beta}]$.
In contrast to \eqref{eq:tran-complexity}, the obtained complexity results indicate that the iterates generated by $\RRM$ can approach the set of stationary points if suitable step size schemes are used. Our bounds eliminate the problematic term depending on $\beta^n$ in \eqref{eq:tran-complexity} and improve the existing complexity results for $\RRHB$ by a factor of $n^{-1/3}$. \\[1mm] 
\noindent\emph{Global Convergence.} We establish an asymptotic global convergence result for $\RRM$. In particular, if the step sizes $\{\alpha_k\}_k$ diminish and satisfy $\sum_{k=1}^\infty \alpha_k^3 < \infty$, we show $\|\nabla f(x^k)\| \to 0$, $k \to \infty$, and hence every accumulation point of the sequence of iterates $\{x^k\}_k$ corresponds to a stationary point of $f$. These findings strengthen the convergence results of Tseng \cite{tseng1998incremental} for $\IGM$. Our theoretical framework allows \emph{arbitrary permutations}, thus encompassing $\IGM$ \cite{tseng1998incremental} as a special case without requiring bounded gradient assumptions. \\[1mm]
\noindent\emph{Iterate Convergence.} When the function $f$ is definable in an o-minimal structure, we establish iterate convergence of $\RRM$, i.e., the entire iterate sequence $\{x^k\}_k$ converges to a stationary point $x^*$ of $f$. This result is shown without invoking standard assumptions such as convexity or coercivity of $f$ or a priori boundedness of $\{x^k\}_k$ (as used in, e.g., \cite{li2023convergence,JMLR:v26:24-1243}). Iterate convergence constitutes a stronger notion of convergence---it eliminates potential iterate oscillations and guarantees \emph{last-iterate} convergence.

\subsection{Notations}
 We use $\N$, $\R$, $\R_+$, $\R_{++}$, and $\R^d$ to denote the set of natural numbers, the real line, the set of nonnegative real numbers, the set of positive real numbers, and the $d$-dimensional Euclidean space, respectively. For any integer $q\in\N$, we set $[q]:=\{1,\ldots,q\}$. 
Unless stated otherwise, $\langle \cdot,\cdot\rangle$ and $\|\cdot\|$ denote the Euclidean inner product and norm. For a vector $v=(v_1,\ldots,v_t)^\top\in\R^t$, we use the standard notation $\|v\|_1:={\sum}_{i=1}^t |v_i|$ and $\|v\|_\infty:=\max_{i=1,\ldots,t}|v_i|$.
For a nonempty set $S\subseteq\Rn$, we write $\dist(x,S):=\inf_{y\in S}\|x-y\|$. For a differentiable function $f:\Rn\to\R$, its set of stationary points is given by $\crit(f):=\{x\in\Rn:\nabla f(x)=0\}$. 
Finally, we assume that there is a sufficiently rich probability space $(\Omega, \mathcal F, \Prob)$ that can describe the stochastic components of $\RRM$ (i.e., the random permutations $\{\pi^k\}_k$) in a unified way. We use $\Exp$ to denote the associated expectation.
\par

\section{Algorithm and assumptions}
\label{sec:convergence analysis}
Mini-batches are frequently used in the training of large-scale models to improve and control stochastic approximation errors. In \Cref{algo:rrm}, we formally introduce a general mini-batch variant of the random reshuffling method with momentum ($\RRM$). For ease of exposition, we assume that the batch size $b\in\N$ is fixed and it holds that $m:=n/b\in\N$.
In addition, we define the mini-batch gradient mapping for each inner iteration $i\in[m]$ and epoch $k\geq 1$ as follows: 
\begin{equation}\label{eq:def-G}
\cG^k_i(x) := \frac1b{\sum}_{j=(i-1)b+1}^{ib} \nabla {f}_{\pi^{k}_{j}}(x).
\end{equation}
Thus, the direction $d_i^k$---computed in step 6 of \Cref{algo:rrm}--- satisfies $d_i^k = \cG^k_i(\nes_i^k)$.

\begin{algorithm2e}[t]
\caption{Random reshuffling with momentum ($\RRM$)}
\label[algorithm]{algo:rrm}
\textbf{Input:} $x^1 = \tilde{x}^1 \in \mathbb{R}^n$, step sizes $\{\alpha_k\}_k$, batch size $b \in \mathbb{N}$, $m = n/b \in \mathbb{N}$, $\beta \in [0,1)$, $\lambda \in [0, \frac{\beta}{1-\beta}]$. \\
\For{$k = 1, 2, \ldots$}{\vspace{1mm}
     Set $y^k_0 = \tilde{x}^k$ and $y^k_1 = x^k$ and generate a permutation $\pi^k=(\pi^k_1,\dots,\pi_n^k)$ of $[n]$\\
    \For{$i = 1, 2, \ldots, m$}{
    $\nes_i^k = y_i^k + \lambda(y_i^k - y_{i-1}^k)$ \hspace{2.55cm} \tcp{extrapolation step}\vspace{1mm}
    $d_i^k = b^{-1} \sum_{j=(i-1)b+1}^{ib} \nabla f_{\pi^k_j}(\nes_i^k)$ \hspace{1.2cm} \tcp{stoch.\ mini-batch gradient}\vspace{1mm}
    $y_{i+1}^k = y_i^k - \alpha_k d_i^k + \beta(y_i^k - y_{i-1}^k)$ \hspace{.9cm} \tcp{momentum step}\vspace{1mm}
    }
    Set $\tilde{x}^{k+1} = y^k_m$ and $x^{k+1} = y_{m+1}^k$}
\end{algorithm2e}

In the context of machine learning, the outer loop (indexed with $k$) corresponds to the overall number of epochs. Each component function typically represents a loss function associated with an individual data point. In each epoch $k$, the permutation $\pi^k = (\pi_1^k,\dots,\pi_n^k)$ shuffles the data points $[n] = \{1,\dots,n\}$, and the algorithm computes stochastic gradients according to the shuffled order $\pi^k$ of the data points. Therefore, each data point is visited sequentially once per epoch before a new random permutation is generated, and the mini-batches in \Cref{algo:rrm} are formed by consecutive blocks of this shuffled order.

This epoch-wise shuffling strategy is precisely the \emph{without-replacement} sampling, which is standard in practical implementations of stochastic optimization for large-scale learning problems \cite{bottou2012,tseng1998incremental,tensorflow2015,paszke2019pytorch}. Different from the classical analysis of with-replacement sampling schemes, the resulting stochastic gradients are generally biased and dependent across the inner iterations of an epoch. This is one of the main technical difficulties in the study of shuffling methods \cite{gurbu2019,mishchenko2020random,nguyen2021unified}.

As mentioned, $\RRM$ can be viewed as a generalization of popular stochastic momentum algorithms. In the case $\lambda=0$, $\RRM$ corresponds to the $\SGD$ optimizer with the ``\emph{momentum}'' option enabled, cf.\ TensorFlow \cite{tensorflow2015}. When $\lambda =\beta$, $\RRM$ corresponds to the $\SGD$ optimizer with both ``\emph{momentum}'' and ``\emph{nesterov}'' options enabled. If $\lambda = \beta = 0$, then $\RRM$ reduces to $\RR$.

In the following, we outline our main assumptions for the convergence analysis of $\RRM$. 
\begin{assumption} 
	We consider the following conditions:
	\begin{enumerate}[label=\textup{\textrm{(A.\alph*)}},topsep=1ex,itemsep=1ex,partopsep=0ex,leftmargin=8ex]
		\item \label{A1} Each $f_i$, $i \in [n]$, is $\sL$-smooth and bounded from below by $\bar f$. 
        \item \label{A2} The sequence $\{\alpha_k\}_k$ is non-increasing. 
        \item \label{S1} The permutations $\{\pi^k\}_k$ are sampled independently (for each $k$) and uniformly without replacement from $[n]$.
	\end{enumerate}
\end{assumption}

Lipschitz continuity of the gradients $\nabla f_i$, $i \in [n]$, is a common and mild assumption, see \cite{bottou2018optimization,mishchenko2020random,nesterov2018lectures}. Moreover, by \ref{A1} and \eqref{eq:def-G}, the mappings $\cG^k_i$, $i \in [m]$, $k \geq 1$, are $\sL$-continuous.

Condition \ref{A1} also provides a useful upper bound on $\|\nabla f_i(\cdot)\|$, \cite{mishchenko2020random,li2023convergence}; it holds that
\begin{equation}
	\label{eq:Lip-bound}
	\|\nabla f_i(x)\|^2 \leq 2\sL [f_i(x)-\bar f],\quad \forall \; i\in [n].
\end{equation}

\section{Challenges and basic properties}
We now describe the main challenges in studying $\RRM$ and the fundamental distinctions from the updates of $\RR$ and other momentum methods. Our proposed methodological approaches and key properties are presented in \Cref{proxy iterates and Lyapunov sequence}.

\subsection{Challenges and motivation}
\label{subsec:challenges} Expanding the algorithmic update of $\RRM$, we can obtain the following expression for $x^{k+1}-x^k$ (for simplicity, we consider the batch size $b=1$):
\begin{equation*}
x^{k+1} - x^k = - \alpha_k\,{\sum}_{t=1}^n \frac{1-\beta^{n-t+1}}{1-\beta} \cdot \nabla f_{\pi_t^k}(\nes_t^k) + \frac{\beta(1-\beta^n)}{1-\beta}\cdot (x^k-\tilde x^k),
\end{equation*}
(cf. \cref{lemma:update} with $i = m = n$). This reveals several complications that render the standard $\RR$-type and momentum methods' analyses inadequate:
\begin{itemize}[leftmargin=5ex]
    \item \emph{Unequal gradient weighting.} Random reshuffling updates $\{x^k\}_k$ via:
\begin{equation*}
x^{k+1} - x^k = - \alpha_k {\sum}_{t=1}^n \nabla f_{\pi_t^k}(y_t^k)\quad \text{where} \quad y_t^k=x^k - \alpha_k {\sum}_{i=1}^{t-1} \nabla f_{\pi_i^k}(y_i^k).
\end{equation*}
Standard analyses express $\RR$ as a gradient descent method with errors \cite{gurbu2019,mishchenko2020random,nguyen2021unified,li2023convergence}. This is enabled by a structural property: each stochastic gradient $\nabla f_{\pi_t^k}(y_t^k)$ receives \emph{equal weight} and we have $\sum_{t=1}^n\nabla f_{\pi_t^k}(x)=n\nabla f(x)$. Hence, the $\RR$ update can be written as:
\[ x^{k+1} - x^k = -n\alpha_k \nabla f(x^k) + e^k \quad \text{where} \quad e^k:=\alpha_k {\sum}_{t=1}^n [\nabla f_{\pi_t^k}(y_t^k) - \nabla f_{\pi_t^k}(x^k)]. \]
The term $e^k$ aggregates the stochastic errors, and $\|e^k\|^2$ is typically of the order $\cO(n^3\alpha_k^3)$, cf.\ \cite[Lemma 3.2]{li2023convergence}. However, the momentum mechanism in $\RRM$ interferes with this elegant decomposition. The coefficients $\frac{1-\beta^{n-t+1}}{1-\beta}$, $t \in [n]$, introduce \emph{non-uniform weights} and the $\RRM$ update cannot be directly expressed as a gradient descent step plus controllable errors. 
\item \emph{Structural difference from classical momentum.} Traditional momentum methods (including $\SGD$-based momentum methods) admit updates of the form 
\[x^{k+1} - x^k = -\alpha_k d^k + \beta (x^k - x^{k-1})\quad \text{for some $\beta \in [0,1)$ and a direction $d^k$.}\] 
This structure is crucial as it enables the following standard steps: by introducing proxy iterates $z^k:=(x^{k+1}-\beta x^k)/(1-\beta)$, the momentum recursion can be transformed into the tractable iteration $z^{k+1}=z^k-\alpha_k d^k/(1-\beta)$, for which well-established convergence techniques are available and applicable \cite{ghadimi2015globalheavyball,yu2019linear,liu2020improved,qiu2024convergence}. By contrast, $\RRM$'s momentum update involves the term $x^k - \tilde{x}^k$, where $\tilde{x}^k$ represents the \emph{second-to-last} inner iterate from the previous epoch rather than $x^{k-1}$. As a result, the existing analysis strategies, \cite{ghadimi2015globalheavyball,liu2020improved,yu2019linear,qiu2024convergence}, cannot be directly applied to $\RRM$. The key issue is that $\tilde{x}^k$ depends on the entire sequence of inner updates within epoch $k-1$, which creates a complex dependence on all intermediate gradient evaluations. This more intricate coupling currently cannot be captured by the standard proxy iterate transformation. 
\item \emph{Biased gradient evaluations and extrapolation.} It is well-known that stochastic gradients generated through sampling without replacement are \emph{biased estimates} of the true gradient, i.e., we generally have $\mathbb{E}[\nabla f_{\pi_t^k}(x)|x] \neq \nabla f(x)$, cf.\ \cite{gurbu2019,mishchenko2020random}. Moreover, in our setting, the stochastic gradients are evaluated at the extrapolated points $\nes_t^k = y_t^k + \lambda(y_t^k - y_{t-1}^k)$ rather than at the current iterates $y_t^k$. This extrapolation introduces additional bias and creates momentum-dependent dynamics in the stochastic gradient estimates. 
\end{itemize}

These differences necessitate a more nuanced analysis to handle the interplay between momentum dynamics, unequal weighting, and extrapolation---challenges that are currently not addressed in existing $\RR$ and/or momentum method analyses \cite{ghadimi2015globalheavyball,gurbu2019,liu2020improved,mishchenko2020random,nguyen2021unified,qiu2024convergence,yu2019linear}.

\subsection{Proxy iterates, Lyapunov sequence, and descent-type property}\label{proxy iterates and Lyapunov sequence}

To handle the challenges posed by the non-standard momentum structure of $\RRM$ and to disentangle the different coupled terms, we introduce the following proxy iterates.

\begin{definition}[Proxy iterates]
The proxy iterates are defined by
\begin{equation} \label{eq:def-z}
z^k := \frac{1}{1-\beta} \cdot x^k - \frac{\beta}{1-\beta} \cdot  \tilde x^k.
\end{equation}
\end{definition}

Applying the update rule of $\RRM$ and invoking \eqref{eq:def-z}, we can infer
\begin{equation} \label{eq:zk-works}
\begin{aligned} 
    (1-\beta) z^{k+1} &= x^{k+1} - \beta \tilde x^{k+1} = y_{m+1}^k - \beta y_m^k = y_{m}^k - \beta y_{m-1}^k - \alpha_k d_{m}^k =  \, \dots \\
    &= y_1^k - \beta y_0^k - \alpha_k {\sum}_{i=1}^m d^k_i  = x^k - \beta \tilde x^k - \alpha_k {\sum}_{i=1}^m d^k_i  = (1-\beta) z^k - \alpha_k {\sum}_{i=1}^m d^k_i.
\end{aligned}
\end{equation}
Thus, the proxy iterates $\{z^k\}_k$ manage to capture the momentum-based dependencies in $\RRM$ similar to the classical auxiliary variables used in the analysis of deterministic and stochastic gradient descent-type momentum methods. To facilitate our convergence analysis, we introduce a Lyapunov sequence that combines the proxy function value $f(z^k)$ and the distance term $\|z^k-x^k\|^2$.


\begin{definition}[Lyapunov sequence]
\label{def:L}
The Lyapunov sequence is defined by
\begin{equation}
	\label{eq:define-L}
	\lya_{k}:= [f(z^{k}) - \bar f] + \sH \alpha_{k}\|z^{k} - x^{k}\|^2\quad \text{where} \quad \sH:=\frac{9\sL^2m}{8(1-\beta)(1-\beta^m)}.
\end{equation}
\end{definition}

For later use, let us further define
\[ \sD:= \frac{1}{(1-\beta^m)^2}\Big[\frac{\sL}{1-\beta} \Big]^3 \qquad \text{and}\qquad \Delta(t):=\lya_1 \cdot \exp(\sD t). \]

Now, we are ready to establish an approximate descent property for $\{\lya_k\}_k$.
\begin{proposition}[Approximate descent property] \label{prop:approximate descent}
Assume \ref{A1}--\ref{A2} hold. Let $\{x^k\}_k$ be generated by $\RRM$ with $\beta\in[0,1)$, $\lambda \in [0, \frac{\beta}{1-\beta}]$, and step sizes $\alpha_k \in (0,\frac{(1-\beta)(1-\beta^m)}{4\sL m}]$, and let $\{z^k\}_k$ be defined as in \eqref{eq:def-z}. Let $T\geq 1$ denote the total number of epochs. 
\begin{enumerate}[label=\textup{(\alph*)},topsep=1ex,itemsep=0ex,partopsep=0ex, leftmargin = 5ex]
\item For all $1\leq k \leq T$, it holds that $\lya_{k} \leq \Delta ({\sum}_{i=1}^T  m^3\alpha_i^3)$ and 
\begin{equation*}
\begin{aligned} \lya_{k+1} & \leq \lya_k +  \Delta\Big(m^3{\sum}_{i=1}^T  \alpha_i^3\Big) \sD m^3\alpha_k^3 \\ & \hspace{4ex} - \frac{1-\beta}{4m\alpha_k}\|z^{k+1} - z^{k}\|^2 - \frac{m\alpha_k}{4(1-\beta)}\Big[\frac14\|\nabla f(x^k)\|^2 +\frac15 \|\nabla f(z^k)\|^2\Big]. \end{aligned}    
\end{equation*}
\item If, in addition, \ref{S1} is satisfied, then for all $1\leq k \leq T$, we have 
\[
    \Exp[\lya_{k+1}]\leq \Exp[\lya_k] - \frac{m\alpha_k}{16(1-\beta)} \Exp[\|\nabla f(x^k)\|^2] +  \Delta\Big(\frac{m^2}{b}{\sum}_{i=1}^T  \alpha_i^3\Big) \cdot \frac{\sD m^2}{b}\alpha_k^3.
\]
\end{enumerate}
\end{proposition}

\begin{proof}[Proof sketch.]
The proof of \Cref{prop:approximate descent} relies on several preparatory estimates derived in \Cref{sec:key-lemmas}, and the full proof is presented in \Cref{subsec:proof approximate descent}. We now briefly sketch some of the core ideas. Using the Lipschitz continuity of $\nabla f$ and algorithmic bounds, we can show:
\begin{equation} \label{eq:roadmap}
    \begin{aligned}
	f(z^{k+1}) - \bar f & \leq f(z^k) - \bar f - p_1 \alpha_k \|\nabla f(z^k)\|^2 + p_2\alpha_k {\sum}_{i=1}^m\|\nes_i^k-z^k\|^2 \\ 
    & \leq [1+q_1\alpha_k^3][f(z^k)-\bar f] - q_2\alpha_k \|\nabla f(z^k)\|^2 + q_3 \alpha_k \|x^k-z^k\|^2,
\end{aligned}
\end{equation}
where $p_1,p_2,q_1,q_2,q_3 >0$ are suitable constants (cf.\ \Cref{lemma:sum y-z} and \Cref{subsec:proof approximate descent}). Furthermore, it is possible to derive a recursive bound for the error term $\|x^k-z^k\|^2$: 
\[\|x^{k+1} - z^{k+1}\|^2 \leq \eta \|x^k-z^k\|^2 + s_1 \alpha_k^2\|\nabla f(z^k)\|^2 + s_2 \alpha_k^2[f(z^k) - \bar f], \]
where $\eta\in(0,1)$ and $s_1, s_2 >0$ (see \Cref{lemma:sum z-x}). Due to $\eta < 1$, these two estimates can be combined to balance the term $\|x^k-z^k\|^2$ in \eqref{eq:roadmap}. This motivates the definition of the Lyapunov sequence $\{\lya_k\}_k$ and enables the descent properties stated in \Cref{prop:approximate descent}. Our analysis and the sketched bounds rely heavily on the proxy iterates $\{z^k\}_k$ and on the behavior of the proxy function values $\{f(z^k)\}_k$ (rather than $\{f(x^k)\}_k$).
\end{proof}
\section{Iteration complexity}\label{sec:comp}
Leveraging the approximate descent property, we can now establish new iteration complexity results for $\RRM$. More detailed comparisons with other existing complexity bounds are provided in \cref{table:super-nice}; see also \Cref{rem:discussion} for further discussions.

\begin{theorem}[Complexity bounds]
\label{thm:complexity}
We assume that \ref{A1}--\ref{A2} hold. Let $\{x^k\}_k$ be generated by $\RRM$ with $\alpha_k \in (0, \frac{(1-\beta)(1-\beta^m)}{4\sL m}]$, $\beta\in[0,1)$, $\lambda\in[0,\frac{\beta}{1-\beta}]$. Let $T\geq 1$ denote the total number of epochs.
\begin{enumerate}[label=\textup{(\alph*)},topsep=1ex,itemsep=1ex,partopsep=0ex, leftmargin = 5ex]
\item If $\sum_{k=1}^T \alpha_k^3 \leq \frac{1}{\sD m^3}$, then $\min_{k=1, \ldots, T} \|\nabla f(x^k)\|^2 \leq \frac{1+3\sD m^3 \sum_{k=1}^T \alpha_k^3}{m\sum_{k=1}^T \alpha_k} \cdot 16(1-\beta) [f(x^1) - \bar f]$.
\item If, in addition, \ref{S1} is satisfied and $\sum_{k=1}^T \alpha_k^3 \leq \frac{b}{\sD m^2}$, we then have
 \[
 \min_{k=1, \ldots, T} \Exp [\|\nabla f(x^k)\|^2] \leq \frac{1+3\sD m^2 b^{-1} \sum_{k=1}^T \alpha_k^3}{m\sum_{k=1}^T \alpha_k} \cdot 16(1-\beta) [f(x^1) - \bar f].  
 \]
\end{enumerate}
\end{theorem}

\begin{proof}
Under the stated step size condition, \Cref{prop:approximate descent} is applicable. Summing the recursion in \Cref{prop:approximate descent} (a) for $k=1,\dots,T$ and recalling $\Delta(t):=\lya_1 \cdot \exp(\sD t)$, we have
\begin{equation*}
    \begin{aligned}
      \frac{m}{16(1-\beta)} \; {\sum}_{k=1}^T \alpha_k \|\nabla f(x^k)\|^2  &\leq \lya_1 + \Delta\Big(m^3{\sum}_{i=1}^T\alpha_i^3\Big) \cdot \sD m^3{\sum}_{i=1}^T \alpha_i^3
      \\& = \Big[ 1 + \exp\Big(\sD m^3{\sum}_{i=1}^T \alpha_i^3 \Big) \cdot \sD m^3 {\sum}_{i=1}^T \alpha_i^3 \Big] \cdot \lya_1.
\end{aligned}
\end{equation*}
Using $\sum_{i=1}^T \alpha_i^3 \leq \frac{1}{\sD m^3}$, $\exp(1) \leq 3$, and $\lya_1 = f(x^1)-\bar f$, we obtain 
\begin{equation*}
    \begin{aligned}
        \min_{k=1, \ldots, T} \|\nabla f(x^k)\|^2 &\leq \Big({\sum}_{k=1}^T \alpha_k \|\nabla f(x^k)\|^2\Big)/\Big({\sum}_{k=1}^T \alpha_k\Big)\\ &\leq \frac{1+3\sD m^3 \sum_{k=1}^T \alpha_k^3}{m\sum_{k=1}^T \alpha_k} \cdot 16(1-\beta) [f(x^1) - \bar f].
    \end{aligned}
\end{equation*}
When \ref{S1} is satisfied, then the estimate in \cref{prop:approximate descent} (b) is applicable. Summing this estimate for $k=1,\dots,T$, it follows
\[
\frac{m}{16(1-\beta)} {\sum}_{k=1}^T \alpha_k \Exp[\|\nabla f(x^k)\|^2]  \leq  \Big[ 1 + \exp\Big(\frac{\sD m^2}{b}{\sum}_{i=1}^T\alpha_i^3 \Big) \cdot \frac{\sD m^2}{b} {\sum}_{i=1}^T \alpha_i^3 \Big] \lya_1.
\]
Using $\sum_{i=1}^T \alpha_i^3 \leq \frac{b}{\sD m^2}$, $\exp(1) \leq 3$, and $\lya_1 = f(x^1)-\bar f$, we can now directly repeat the steps from part (a) to complete the proof.
\end{proof}
\begin{corollary}[Complexity bounds: constant step sizes]
\label[corollary]{coro:complexity-constant}
Let $\{x^k\}_k$ be generated by $\RRM$ with $\{\alpha_k\}_k \subset \R_{++}$, $\beta\in[0,1)$, $\lambda\in[0,\frac{\beta}{1-\beta}]$ and let $T\in\N$ be given. Assume that \ref{A1} holds.
\begin{enumerate}[label=\textup{(\alph*)},topsep=1ex,itemsep=1ex,partopsep=0ex, leftmargin = 5ex]
\item Suppose that $\alpha_k = \frac{(1-\beta)(1-\beta^m)\alpha}{\sL m}$, $\alpha \leq \min\{\frac{1}{4}, \frac{1}{[(1-\beta^m)T]^{1/3}}\}$ for all $k$. Then, it holds that
\[
 \min_{k=1,\ldots,T} \|\nabla f(x^k)\|^2 \leq \Big[ \frac{1}{(1-\beta^m)\alpha T} + {3 \alpha^2} \Big] \cdot 16\sL [f(x^1) - \bar f].
\]
\item In addition, under \ref{S1} and if $\alpha_k = \frac{(1-\beta)(1-\beta^m)\alpha}{\sL m}$, $\alpha \leq \min\{\frac{1}{4}, [\frac{n}{(1-\beta^m)T}]^{1/3} \}$, we have
\[
 \min_{k=1,\ldots,T} \Exp[\|\nabla f(x^k)\|^2] \leq \Big[ \frac{1}{(1-\beta^m)\alpha T} + \frac{3 \alpha^2}{n} \Big] \cdot 16 \sL [f(x^1) - \bar f].
\]
\end{enumerate}
\end{corollary}
\begin{proof}
Assumption \ref{A2} clearly holds and we have $\alpha_k \leq \frac{(1-\beta)(1-\beta^m)}{4\sL m}$ and $\sD m^3 \sum_{i=1}^T \alpha_i^3 = (1-\beta^m)T \alpha^3 \leq 1$. Hence, all requirements in \Cref{thm:complexity} (a) are satisfied and substituting $\alpha_k$ into \Cref{thm:complexity} (a) establishes part (a). The proof of part (b) follows similarly. 
\end{proof}

\begin{remark}[Discussion of the complexity results] \label[remark]{rem:discussion}
\;
\begin{itemize}[leftmargin=5ex]
\item Except for the momentum-based factor $\frac{1}{1-\beta^m}$, the bounds in \Cref{coro:complexity-constant} align with the ones obtained in \cite{mishchenko2020random,nguyen2021unified} for $\RR$ and incremental gradient methods. The additional dependence on the momentum parameter $\beta$ is typical for stochastic momentum methods and can also be observed in other works \cite{liu2020improved,tran2021smg,yang2016unified}. In practice, the momentum parameter is often chosen as $\beta \approx 0.9$, which is the default value in PyTorch \cite{paszke2019pytorch} and TensorFlow \cite{tensorflow2015}. In this case, if $m$ is ``large''---say $m \geq 50$---we have $\frac{1}{1-\beta^m} \leq 1.005$; this is close to $1$. 

\item Setting $\alpha=[n/((1-\beta^m)T)]^{1/3}$ in \Cref{coro:complexity-constant} (b), the bound reduces to 
\begin{equation*} 
\min_{k=1,\ldots,T} \Exp[\|\nabla f(x^k)\|^2] = \cO\Big(\frac{1}{(1-\beta^m)^{2/3}T^{2/3} n^{1/3}}\Big)\end{equation*}
provided that $T \geq 64n/(1-\beta^m)$. This allows us to recover the in-expectation complexity results derived in \cite{mishchenko2020random,nguyen2021unified} for $\RR$. Similarly, choosing $\alpha = [(1-\beta^m)T]^{-1/3}$ in \Cref{coro:complexity-constant} (a), we obtain $\min_{k=1,\dots,T} \|\nabla f(x^k)\|^2 = \mathcal O([(1-\beta^m)T]^{-2/3})$ if $T \geq 64/(1-\beta^m)$. Thus, by exploiting the randomness of the uniform sampling scheme, the complexity results for $\RRM$ can be improved by a factor of $n^{1/3}$; see \cite{mishchenko2020random,nguyen2021unified} for related discussions for $\RR$. As a consequence, to reach an $\varepsilon$-accurate solution---$\min_{k=1,\dots, T} \Exp[\|\nabla f(x^k)\|] \leq \varepsilon$---$\RRM$ requires $T = \cO((1-\beta^m)^{-1}n^{-1/2}\varepsilon^{-3})$ epochs. Since there are $n$ gradient evaluations per epoch, the overall complexity of $\RRM$ in terms of gradient evaluations is given by
\begin{equation*}
    \cO((1-\beta^m)^{-1}\sqrt{n}\varepsilon^{-3}) = \cO(\sqrt{n}\varepsilon^{-3}).
\end{equation*}
Thus, compared to the standard $\cO(\varepsilon^{-4})$-complexity of $\SGDM$,  $\RRM$ achieves a better complexity bound whenever $n \lesssim \varepsilon^{-2}$. More detailed comparisons are provided in \Cref{table:super-nice}.
\item To the best of our knowledge, \cref{thm:complexity} provides the first iteration complexity for $\RRM$ in the nonconvex case \emph{without requiring bounded (stochastic) gradients} or a \emph{vanishing} momentum parameter $\beta$. Specifically, our results are valid for arbitrary choices of $\beta\in[0,1)$ and $\lambda \in [0, \frac{\beta}{1-\beta}]$. This flexibility aligns well with common practices and settings.
\end{itemize}
\end{remark}

\setcounter{table}{1}
\renewcommand{\tablename}{Table}
\begin{table}[t]
\centering
{\footnotesize
\setlength{\tabcolsep}{5pt}
\NiceMatrixOptions{cell-space-limits=1pt}
\begin{NiceTabular}{|c|p{2.55cm}p{2.2cm}|p{6.2cm}|c|}%
 [ 
   code-before = 
    \rectanglecolor{lavender!30}{4-2}{4-5}
    \rectanglecolor{lavender!30}{6-2}{6-5}
    \rectanglecolor{lavender!30}{8-2}{8-5}
 ]
\toprule
\Block{3-1}{\textbf{Alg.}} & \Block[c]{1-3}{\textbf{Iteration complexity: nonconvex setting}} & & & \Block{3-1}{\textbf{Ref.}}  \\ \Hline
& \Block{2-1}{no variance cond. $\&$ unbound. grad.} 
& \Block{2-1}{$\beta\in[0,1)$ is free} & \Block{1-1}{$\#$ of grad. evaluations to reach} & \\
& 
&  & \Block{1-1}{$\min_{k=1,\ldots,T}\; \Exp[\|\nabla f(x^k)\|] \leq \varepsilon$} & \\[1mm] \Hline
\Block{1-1}{$\RR$} & \Block{1-1}{\checkmark}  & \Block{1-1}{--} & \Block{1-1}{$\varepsilon^{-3}\cdot \sL\sqrt{n}\max\{\varepsilon \sqrt{n},\sqrt{\sA}+\sB\}$\,${}^{\textcolor{blue}{\text{(a)}}}$}  & \Block{1-1}{\cite{mishchenko2020random}} \\
 \Hline
\Block{1-1}{$\SGDM$} & \Block{1-1}{\xmark} & \Block{1-1}{\checkmark} & \Block{1-1}{$\varepsilon^{-4}\cdot \sL\max\{\frac{\varepsilon^2}{1-\beta},\sB^2\}$\,${}^{\textcolor{blue}{\text{(b)}}}$}  & \Block{1-1}{\cite{liu2020improved}} \\[1mm]
 \Hline
\Block{1-1}{$\SMG$} & \Block{1-1}{\checkmark} & \Block{1-1}{\checkmark} & \Block{1-1}{$\varepsilon^{-3}\cdot\frac{\sL\sqrt{n}}{(1-\beta)^{3/2}}\max\{\varepsilon\sqrt{\sA+n},\sB\}$\,${}^{\textcolor{blue}{\text{(c)}}}$} & \Block{1-1}{\cite{tran2021smg}} \\[1mm]
\Hline
\Block{2-1}{$\RRM$}  & \Block{1-1}{\xmark} & \Block{1-1}{\xmark} &  \Block{1-1}{$\varepsilon^{-3}\cdot \sL{n} (1+\frac{\varepsilon^2}{\sG^2})\max\{\varepsilon,\sG\}$\,${}^{\textcolor{blue}{\text{(d)}}}$} & \Block{1-1}{\cite{tran2021smg}} \\[1mm]
\Hline
 & \Block{1-1}{\checkmark} & \Block{1-1}{\checkmark} & \Block{1-1}{$\varepsilon^{-3}\cdot \frac{\sL\sqrt{n}}{1-\beta^m}\max\{\varepsilon\sqrt{n},\sqrt{\sL}\}$\,${}^{\textcolor{blue}{\text{(e)}}}$} & \Block{1-1}{Ours} \\
\bottomrule
\end{NiceTabular}
}
\vspace{1ex}
\caption{Comparison of complexity bounds for standard momentum-based and shuffling-type methods.}
\label{table:super-nice}
\end{table}

\paragraph{Details on \Cref{table:super-nice}}

In the second column, the icon ``\checkmark'' indicates that no particular variance assumption is required and that the (stochastic) gradients need not be bounded; in the third column, ``\checkmark'' indicates that the momentum parameter $\beta \in [0,1)$ can be chosen freely; and the fourth column reports the number of gradient evaluations $K = T$ for $\SGDM$ and $K = nT$ for $\RR$, $\SMG$, and $\RRM$ required to reach an $\varepsilon$-accurate solution satisfying $\min_{k=1,\dots, T} \Exp[\|\nabla f(x^k)\|] \leq \varepsilon$. Additional information and derivations are provided in \Cref{app:table}.
\begin{enumerate}[label={\textcolor{blue}{\textup{(\alph*)}}},topsep=1ex,itemsep=0ex,partopsep=0ex,leftmargin=5ex]
\item Based on the variance condition $\frac{1}{n} \sum_{i=1}^n \|\nabla f_i(x)-\nabla f(x)\|^2 \leq 2\sA[f(x)-\bar f]+\sB^2$. $\RR$ corresponds to $\RRM$ with $\beta = \lambda =0$.
\item Based on the condition $\frac{1}{n} \sum_{i=1}^n \|\nabla f_i(x)-\nabla f(x)\|^2 \leq \sB^2$.
\item The complexity result in \cite[Theorem 2]{tran2021smg} is based on the variance condition $\frac{1}{n} \sum_{i=1}^n \|\nabla f_i(x)-\nabla f(x)\|^2 \leq \sA\|\nabla f(x)\|^2+\sB^2$. The momentum term of $\SMG$ is fixed within an outer loop (epoch), which is different from the momentum mechanisms used in practice.
\item The bound in \cite[Theorem 3]{tran2021smg} applies to $\RRM$ with $\lambda =0$ and shuffle-once option $\pi^k \equiv \pi$, and is based on the bounded gradient assumption $\|\nabla f_i(x)\| \leq \sG$ for all $x$ and $i$ and sufficiently small momentum $\beta^n \lesssim \varepsilon^2$.
\item As each function $f_i$ is $\sL$-Lipschitz smooth, we have $\frac{1}{n}\sum_{i=1}^{n}\|\nabla f_i(x)-\nabla f(x)\|^2 \leq 2\sL[f(x)-\bar f] - \|\nabla f(x)\|^2 \leq 2\sL[f(x)-\bar f]$. Hence, for comparison, we may set $\sA = \sL$, $\sB = 0$ in \textcolor{blue}{\textup{(a)}} and $\sA = 0$, $\sB \sim \sqrt{\sL}$ in \textcolor{blue}{\textup{(c)}}.
\end{enumerate}
\par

Finally, we specialize \Cref{thm:complexity} to polynomial step sizes.

\begin{corollary}[Complexity bounds: polynomial step sizes]
\label{coro:complexity-polynomial}
Let condition \ref{A1} hold. Let $\{x^k\}_k$ be generated by $\RRM$ with $\{\alpha_k\}_k \subset \R_{++}$, $\beta\in[0,1)$, $\lambda\in[0,\frac{\beta}{1-\beta}]$ and let $T\in\N$ be given.
\begin{enumerate}[label=\textup{(\alph*)},topsep=1ex,itemsep=0ex,partopsep=0ex, leftmargin = 5ex]
\item Consider $\alpha_k = \frac{(1-\beta)(1-\beta^m)}{\sL m}\frac{\alpha}{k^\gamma} $ with $\gamma \in (\frac13, 1)$, $\alpha \leq \min\{\frac{1}{4}, [\frac{3\gamma - 1}{3\gamma(1-\beta^m)}]^{1/3} \}$. Then, we have
\[
 \min_{k=1,\ldots,T} \|\nabla f(x^k)\|^2 \leq \Big[ \frac{1}{(1-\beta^m)\alpha} +  \frac{9\gamma \alpha^2}{3\gamma-1} \Big] \cdot \frac{16\sL(1-\gamma) [f(x^1) - \bar f]}{(T+1)^{1-\gamma} - 1}.
\]
\item If \ref{S1} holds and $\alpha_k = \frac{(1-\beta)(1-\beta^m)}{\sL m}\frac{\alpha}{k^\gamma}$ with $\gamma \in (\frac13, 1)$, $\alpha \leq \min\{\frac{1}{4}, [\frac{(3\gamma - 1)n}{3\gamma(1-\beta^m)}]^{1/3}\}$, then:
\[
 \min_{k=1,\ldots,T} \Exp[\|\nabla f(x^k)\|^2] \leq \Big[ \frac{1}{(1-\beta^m)\alpha} + \frac{9\gamma \alpha^2}{(3\gamma-1)n} \Big] \cdot \frac{16\sL(1-\gamma) [f(x^1) - \bar f]}{(T+1)^{1-\gamma} - 1}.    
\]
\end{enumerate}
\end{corollary}
\begin{proof} Setting $\delta := \frac{(1-\beta)(1-\beta^m)}{\sL m}$ and using the integral test and $\gamma \in (\frac13,1)$, we have 
\[ \sD m^3 {\sum}_{k=1}^T \alpha_k^3 \leq \sD m^3 \delta^3 \alpha^3 \Big[1+\int_{1}^T \frac{1}{x^{3\gamma}}\,\mathrm{d}x\Big] \leq (1-\beta^m)\alpha^3 \cdot \frac{3\gamma}{3\gamma-1} \]
and $\sum_{k=1}^T \alpha_k \geq \frac{\delta\alpha}{1-\gamma}[(T+1)^{1-\gamma}-1]$. Hence, by assumption, \Cref{thm:complexity} (a) is applicable, which allows us to establish the desired complexity bound. Part (b) follows similarly.
\end{proof}

\begin{remark}[Optimal choice of the polynomial step sizes]
\label{rem:choice-gamma}

Corollary \ref{coro:complexity-polynomial} suggests an epoch-dependent choice of the polynomial step size parameters to achieve the optimal complexity. In particular, for $\gamma = \frac13$, we have $Dm^3\sum_{k=1}^T \alpha_k^3 \leq \frac{(1-\beta^m)\alpha^3}{n}[1+\log(T)]$. Hence, in the scenario $\alpha \sim [\frac{n}{(1-\beta^m)\log(T+1)}]^{1/3}$ and following the earlier derivations, we can obtain 
\[ \min_{k=1,\ldots,T} \Exp[\|\nabla f(x^k)\|^2]
    = \cO\Big(\frac{\log(T+1)^{1/3}}{(1-\beta^m)^{2/3}n^{1/3}(T+1)^{2/3}}\Big)
\]
provided that $T$ is sufficiently large. The dependence on $T$ can be removed by, e.g., considering more general step sizes of the form $\alpha_k = \frac{\alpha\delta}{(k\log(k))^\gamma}$ where $\gamma = \frac13$, $\delta = \frac{(1-\beta)(1-\beta^m)}{Lm}$, and $\alpha \in (0,\frac14]$. We will omit explicit computations for such choice here. 

\end{remark}

\section{Global convergence}
Due to the $\min$-operation, the bounds in Corollaries \ref{coro:complexity-constant} and \ref{coro:complexity-polynomial} do not directly imply $\|\nabla f(x^k)\| \to 0$ (when taking $T\to\infty$). 
    Let $\{x^k\}_k$ be generated by $\RRM$. We define the set of accumulation points of $\{x^k\}_k$ by
%
\begin{equation} \label{eq:limit point set}
\cA := \{ x \in \Rn:{\liminf}_{k\to\infty}\|x^k-x\|=0\}.
\end{equation}
In the following, we complement our non-asymptotic results and discuss the asymptotic convergence behavior of $\RRM$. 

\begin{theorem}[Global convergence] \label{thm:global_convergence}
	Let \ref{A1}--\ref{A2} hold and let $\{x^k\}_k$ be generated by $\RRM$ with $\beta\in[0,1)$, $\lambda \in [0, \frac{\beta}{1-\beta}]$, and step sizes $\{\alpha_k\}_k \subseteq \R_{++}$ satisfying 
 \begin{equation}\label{step-size-1}
   \alpha_k \leq \frac{(1-\beta)(1-\beta^m)}{4\sL m}, \quad  {\sum}_{k=1}^{\infty} \alpha_k = \infty,\quad \text{and} \quad {\sum}_{k=1}^{\infty}\alpha_k^3 < \infty.
 \end{equation}
	Furthermore, let $\{z^k\}_k$ be given as in \eqref{eq:def-z}. Then, the following statements are valid:
	\begin{enumerate}[label=\textup{(\alph*)},topsep=1ex,itemsep=0.5ex,partopsep=0ex, leftmargin = 5ex]
     \item We have $\sum_{k=1}^\infty \alpha_k \|\nabla f(x^k)\|^2 < \infty$ and $\min_{k=1,\dots,T} \|\nabla f(x^k)\|^2 = o\big(({\sum_{k=1}^T \alpha_k})^{-1}\big)$, $T \to \infty$.
	\item It holds that $\lim_{k\to\infty} \|\nabla f(x^k)\|=0$ and $\lim_{k\to\infty} \|\nabla f(z^k)\|=0$, i.e., every accumulation point of $\{x^k\}_k$ and $\{z^k\}_k$ is a stationary point of $f$.
	\item The sequences $\{f(x^k)\}_k$ and $\{f(z^k)\}_k$ converge to some constant $f^*\in\R$.
	\end{enumerate}
\end{theorem}
\begin{remark} \label{rem:global_convergence}
We note that the results in \Cref{thm:global_convergence} hold \emph{surely} for every sequence of permutations $\{\pi^k\}_k$. Furthermore, conditions of the form \eqref{step-size-1} appear frequently in the analysis of shuffling methods; see \cite{nguyen2021unified,li2023convergence}. The requirements in \eqref{step-size-1} are satisfied, e.g., for polynomial step sizes $\alpha_k \sim {k^{-\gamma}}$ with $\gamma\in(\frac13,1]$. 
\end{remark}
 
\begin{proof}[Proof of \texorpdfstring{\cref{thm:global_convergence}}{Theorem 5.1}]
Let us define 
\begin{equation}
    \label{eq:def G and u}
    \sG:=\Delta\Big(m^3{\sum}_{k=1}^{\infty}\alpha_k^3\Big) \cdot \sD m^3 \quad\text{and}\quad u_k := \sG\; {\sum}_{i=k}^\infty \alpha_i^3.
\end{equation}
Due to ${\sum}_{k=1}^{\infty}\alpha_k^3<\infty$, we have $\sG < \infty$ and $u_k\to 0$. Applying \cref{prop:approximate descent} (a) with $T = \infty$, it follows
\begin{equation}
    \label{eq:thm:approximate descent}
    \lya_{k+1} + u_{k+1} \leq \lya_k + u_k - \frac{1-\beta}{4m\alpha_k}\|z^{k+1} - z^{k}\|^2 - \frac{m\alpha_k}{4(1-\beta)}\Big[\frac{\|\nabla f(x^k)\|^2}{4} +\frac{\|\nabla f(z^k)\|^2}{5} \Big].
\end{equation}
Clearly, the sequence $\{\lya_k + u_k\}_k$ is non-increasing and bounded from below by \ref{A1}. Combining this with $u_k\to0$, we can infer that $\{\lya_k\}_k$ converges to some $f^*\in\R$. Hence, summing \eqref{eq:thm:approximate descent} for $k \geq 1$, we have 
\begin{equation}
    \label{eq:global-core}
    {\sum}_{k=1}^\infty \alpha_k \|\nabla f(x^k)\|^2 < \infty, \quad {\sum}_{k=1}^\infty \alpha_k \|\nabla f(z^k)\|^2 < \infty, \quad {\sum}_{k=1}^\infty  \frac{\|z^{k+1}-z^{k}\|^2}{\alpha_k} < \infty.
\end{equation}
Thanks to \eqref{eq:global-core}, Kronecker's lemma (\Cref{lemma:kronecker}) is applicable with $r_k := \alpha_k\|\nabla f(x^k)\|^2$ and $s_k := \sum_{i=1}^k \alpha_i$ (by \eqref{step-size-1}, $\{s_k\}_k$ is non-decreasing with $s_k \to \infty$). This yields 
\[ s_T^{-1} \cdot {\sum}_{k=1}^T \alpha_k \Big({\sum}_{i=1}^k \alpha_i\Big) \|\nabla f(x^k)\|^2 \to 0. \] 
Consequently, due to $\sum_{k=1}^T \alpha_k (\sum_{i=1}^k \alpha_i) = \frac12((\sum_{k=1}^T \alpha_k)^2 + \sum_{k=1}^T \alpha_k^2) = \Theta((\sum_{k=1}^T \alpha_k)^2)$, it follows $\min_{k=1,\dots,T} \|\nabla f(x^k)\|^2 = o\big((\sum_{k=1}^T \alpha_k)^{-1}\big)$. This finishes the proof of part (a).

Next, we verify the convergence of the sequences $\{f(z^k)\}_k$ and $\{\|\nabla f(z^k)\|\}_k$. We then transfer these results to $\{f(x^k)\}_k$ and $\{\|\nabla f(x^k)\|\}_k$ using the following auxiliary lemma. 


\begin{lemma}[Distance between $\{x^k\}_k$ and $\{z^k\}_k$]
\label{prop:distance}
Let the assumptions \ref{A1}--\ref{A2} hold and let $\{x^k\}_k$ be generated by $\RRM$ with $\beta\in[0,1)$, $\lambda \in [0, \frac{\beta}{1-\beta}]$, and step sizes $\{\alpha_k\}_k \subseteq \R_{++}$ satisfying \eqref{step-size-1}. Let $\{z^k\}_k$ be defined as in \eqref{eq:def-z}. Then, it holds that $\|z^k-x^k\| \to 0$ as $k \to \infty$.
%
\end{lemma}
\begin{proof}
By \cref{lemma:sum z-x} (a) (see Appendix \ref{sec:key-lemmas}) and using $\|\nabla f(z^k)\|^2 \leq 2\sL [f(z^k)-\bar f]$ (cf.\ \eqref{eq:Lip-bound}), we have
    \[
\|z^{k+1} - x^{k+1}\|^2 \leq \eta \|z^k  -x^k\|^2 + \frac{15\sL\beta^2m^2[f(z^k) - \bar f]}{(1-\beta)^2(1-\beta^m)}\cdot \alpha_k^2, \quad \eta \in (0,1),
\]
for all $k$. The assumptions in \Cref{prop:distance} imply that our preceding derivations are valid. In particular, $\{\mathcal R_k\}_k$ is bounded, and hence $\{f(z^k)\}_k$ is also a bounded sequence. Thus, there is $\tilde{\sG}>0$ such that $\|z^{k+1} - x^{k+1}\|^2 \leq \eta \|z^k  -x^k\|^2 + \tilde{\sG}\alpha_k^2$. Due to $\alpha_k \to 0$ and by \cite[Proposition A.30]{ber16}, the sequence $\{\|x^k-z^k\|\}_k$ converges to zero.
\end{proof}

We now continue with the analysis of $\{\|\nabla f(z^k)\|\}_k$. The relation \eqref{eq:global-core} and $\sum_{k=1}^\infty \alpha_k = \infty$ imply $\liminf_{k\to\infty} \|\nabla f(z^k)\| = 0$. Next, we show $\lim_{k\to\infty} \|\nabla f(z^k)\| = 0$ by contradiction. Let us assume that $\limsup_{k\to\infty} \|\nabla f(z^k)\| > 0$. Then, there exist $\varepsilon>0$ and infinite subsequences $\{t_j\}_{j}$ and $\{\ell_j\}_{j}$ such that $t_j<\ell_j<t_{j+1}$,
	\begin{equation}\label{eq:construct subsequence}
		\|\nabla f({z}^{t_j})\|\geq 2\varepsilon,\quad \|\nabla f({z}^{\ell_j})\|<\varepsilon,\quad\text{and}\quad\|\nabla f({z}^{k})\|\geq\varepsilon, 
	\end{equation}
    for all $k=t_j,\dots,\ell_j-1$. 
	Combining \eqref{eq:global-core} and \eqref{eq:construct subsequence}, we can infer $\infty>  {\sum}_{k=1}^{\infty}\alpha_{k}\|\nabla f(z^{k})\|^2 \geq \varepsilon^2{\sum}_{j=1}^{\infty}{\sum}_{k=t_j}^{\ell_j-1}\alpha_{k}$, which yields $\lim_{j\rightarrow \infty} \nu_j = 0$ where $\nu_j := {\sum}_{k=t_j}^{\ell_j-1}\alpha_{k}$.
	Furthermore, applying the triangle and the Cauchy-Schwarz inequality, it follows that
    \begin{equation*}
        \begin{aligned}
            \|{z}^{\ell_j} - {z}^{t_j}\|  & \leq {\sum}_{k=t_j}^{\ell_j-1}{\alpha_{k}^{1/2}} \cdot [\alpha_k^{-1/2}\|{z}^{k+1}-z^{k}\|] \\ & \leq \Big[ {\sum}_{k=t_j}^{\ell_j-1} \alpha_{k} \cdot {\sum}_{k=t_j}^{\ell_j-1} \alpha_k^{-1}\|z^{k+1}-z^k\|^2\Big]^{1/2} \leq \sqrt{\nu_j} \cdot \Big[ {\sum}_{k=1}^{\infty} \alpha_k^{-1}\|z^{k+1}-z^{k}\|^2\Big]^{1/2}. 
        \end{aligned}
    \end{equation*}
    Hence, due to \eqref{eq:global-core} and $\nu_j \to 0$, we have $\|z^{\ell_j}-z^{t_j}\| \to 0$ as $j \to \infty$. 
	%
	Following the construction in \eqref{eq:construct subsequence} and using the Lipschitz continuity of $\nabla f$ and the triangle inequality, it holds that
\[		\varepsilon\leq |\|\nabla f({z}^{\ell_j})\|-\|\nabla f({z}^{t_j})\| | \leq \|\nabla f({z}^{\ell_j})-\nabla f({z}^{t_j})\|\leq \sL\|{z}^{\ell_j}-{z}^{t_j}\| \to 0, \quad j \to \infty.
\]
This is a contradiction and establishes $\lim_{k\to\infty}\|\nabla f(z^k)\|=0$. Moreover, invoking \cref{prop:distance}, we have $\|z^k - x^k\|\to 0$ and $\|\nabla f(x^k)\| \leq \sL\|x^k - z^k\| + \|\nabla f(z^k)\| \to 0$ as $k \to \infty$.

Finally, we discuss the convergence of the sequences $\{f(z^k)\}_k$ and $\{f(x^k)\}_k$. Combining $\lya_k \to f^*$ and $\|z^k - x^k\|\to 0$, we conclude $f(z^k) \to f^*$, $k \to \infty$, by the definition of $\lya_k$ (cf.\ \cref{def:L}). By the Lipschitz continuity of $\nabla f$ and Young's inequality, we further obtain
\begin{equation*}
\begin{aligned}
    |f(x^k) - f(z^k)| &\leq \max\{|\langle \nabla f(x^k), x^k-z^k\rangle|,|\langle \nabla f(z^k), x^k-z^k\rangle|\} + \frac{\sL}{2}\|x^k-z^k\|^2\\
	& \leq \frac{1}{2\sL}\max\{\|\nabla f(x^k)\|^2,\|\nabla f(z^k)\|^2\} + \sL\|x^k-z^k\|^2.
\end{aligned}
\end{equation*}
Thus, due to $\|\nabla f(x^k)\| \to 0$, $\|\nabla f(z^k)\| \to 0$, $\|x^k-z^k\|\to0$, and $f(z^k) \to f^*$ as $k\to\infty$, we can infer $|f(x^k) - f^*| \leq |f(z^k) - f^*| + |f(x^k) - f(z^k)| \to 0$. 
\end{proof}

\section{Iterate convergence}
In this section, we provide (last-)iterate convergence guarantees for $\RRM$ when the objective function is definable in an o-minimal structure; see \cite{AttBolSva13,kur98,lojasiewicz1993,van1998tame}.

For completeness, we briefly recall relevant terminologies. An o-minimal structure on the real field is a sequence of Boolean algebras of subsets of Euclidean spaces that is closed under Cartesian products and coordinate projections, contains all algebraic sets, and whose one-dimensional sets are finite unions of points and intervals. A set is called definable if it belongs to such a structure, and a function is called definable if its graph is definable; see, e.g., \cite{van1998tame,coste2000introduction}.
\par
The class of definable functions is extensive and encompasses a wide array of objective functions relevant to practical applications. In particular, proper closed semi-algebraic functions 
and functions in the log-exp structure \cite{van1998tame} are definable; cf.\  \cite{kur98,BolDanLew06,AttBolRedSou10} for more detailed discussions. Prominent examples include, e.g., logistic regression \cite{li2018calculus}, principal component analysis \cite{liu2019quadratic}, polynomial optimization \cite{DAcuntoKurdyka2005}, and deep neural networks \cite{davis2020stochastic}. 

\begin{assumption}
	We consider the following assumptions:
	\begin{enumerate}[label=\textup{\textrm{(B.\alph*)}},topsep=1ex,itemsep=1ex,partopsep=0ex,leftmargin=8ex]
		\item \label{B2} The function $f$ is definable in an o-minimal structure. 
        \item \label{B1} We have $\liminf_{k\to \infty} \|x^k\| < \infty$ or, equivalently, the set $\cA$ defined in \eqref{eq:limit point set} is non-empty.
\end{enumerate}
\end{assumption}

Our analysis is based on the global convergence results (\Cref{thm:global_convergence}) and the well-known Kurdyka-{\L}ojasiewicz (KL) inequality, \cite{loj63,kur98,AttBolRedSou10}. KL-based analysis techniques have become a key and highly successful tool to establish iterate convergence in nonconvex optimization \cite{absil2005convergence,AttBolSva13,BolSabTeb14,li2015global}. We now present a KL property for definable functions that is particularly suited for stochastic settings. \Cref{lem:KL} is a simplified version of \cite[Lemma 4.11]{josz2024proximal} by Josz et al.; see also \cite[Lemma 4.1]{qiu2025normal} for comparison.

\begin{lemma}[KL property]
\label{lem:KL}
    Let \ref{B2} hold and let $x^*\in \crit(f) := \{x \in \Rn: \nabla f(x) = 0\}$ be given. For all $\vartheta\in(0,1)$, there are $\sC > 0$, $\eta\in(0,1]$, a neighborhood $U$ of $x^*$, and a continuous, concave function $\varrho:[0,\eta) \to \R_+$, which is continuously differentiable on $(0,\eta)$, satisfying 
    \begin{equation}
         \label{eq:desingularization function}
         \varrho(0)=0\quad \text{and} \quad 
    1/\varrho^\prime(s+t) \leq 1/\varrho^\prime(s) + \sC t^{\vartheta} \quad \text{for all $s,t >0$ with $s+t<\eta$},
    \end{equation}
     such that for all $x\in U \cap\{x:0<|f(x)-f(x^*)|<\eta\}$ the following KL inequality holds:
     \begin{equation}
         \label{eq:mer-kl}
         \varrho^\prime(|f(x)-f(x^*)|) \cdot  \|\nabla f(x)\|\geq 1.
     \end{equation} 
\end{lemma}
As mentioned, the KL inequality \eqref{eq:mer-kl} is an important geometric concept which can be leveraged in the analysis of algorithms applied to nonconvex problems, \cite{absil2005convergence,AttBolSva13,BolSabTeb14,li2015global}. 

If the iterates $\{x^k\}_k$ are bounded ($\limsup_{k\to\infty} \|x^k\| < \infty$) and if \eqref{eq:mer-kl} holds at every point $x^* \in \cA$, then the so-called uniformized KL property, \cite[Lemma 6]{BolSabTeb14}, can be used to significantly simplify the discussion. Hence, boundedness of the iterates $\{x^k\}_k$ appears as a prominent assumption in KL-based convergence analyses, see \cite{AttBolRedSou10,BolSabTeb14,li2023convergence,pock2016inertial,tadic2015convergence}.
In the following, we provide (last)-iterate convergence guarantees for $\RRM$ under the weaker assumption \ref{B1}.

\begin{theorem}[Iterate convergence and finite length] \label{thm:iter_convergence}
	Let \ref{A1}--\ref{A2} and \ref{B2}--\ref{B1} hold and let $\{x^k\}_k$ be generated by $\RRM$ with parameters satisfying $\beta\in[0,1)$, $\lambda\in[0,\frac{\beta}{1-\beta}]$, and
 \begin{equation}
     \label{thm:iter step size}
\{\alpha_k\}_k \subset \R_{++}, \quad {\sum}_{k=1}^{\infty} \alpha_k = \infty,\quad {\sum}_{k=1}^{\infty} \alpha_k^{1+\xi} < \infty, \quad \text{for some $\xi\in(0,1)$.}
    \end{equation}
 \begin{enumerate}[label=\textup{(\alph*)},topsep=1ex,itemsep=0.5ex,partopsep=0ex, leftmargin = 5ex]
	\item The iterates $\{x^k\}_k$ converge to some $x^*\in\crit(f)$ and it holds that $\sum_{k=1}^\infty \|z^{k+1}-z^k\| < \infty$, where $\{z^k\}_k$ is defined as in \eqref{eq:def-z}.
    \item We have $\sum_{k=1}^\infty \alpha_k \|\nabla f(x^k)\| < \infty$ and $\min_{k=1,\dots,T} \|\nabla f(x^k)\|^2 = o\big(({\sum_{k=1}^T \alpha_k})^{-2}\big)$, $T \to \infty$.
	\end{enumerate}
\end{theorem}

\begin{remark}[Iterate convergence for stochastic methods]
Most of the existing last-iterate results for stochastic algorithms are limited to convex problems; see, e.g., \cite{harvey2019tight,sebbouh2021almost,liu2022almost}. Moreover, successful applications of KL-type strategies in the context of nonconvex stochastic optimization still remain fairly rare. To our knowledge, the first comprehensive KL-based convergence analysis for $\SGD$ is presented in \cite{tadic2015convergence}, showing (last-)iterate convergence under standard variance assumptions. Recently, in \cite{dereich2021convergence,qiu2024convergence}, the results in \cite{tadic2015convergence} have been extended to momentum variants of $\SGD$ using relaxed techniques and weaker step size requirements.  For shuffling-type methods, existing KL-based analyses appear to be limited to $\RR$ \cite{li2023convergence} and $\SMG$ \cite{LiaXu24}. Both works \cite{li2023convergence,LiaXu24} establish iterate convergence if the underlying objective function satisfies a quasi-additive-type KL or the stronger {\L}ojasiewicz inequality and require a priori boundedness of the iterates or coercivity conditions. In \cref{thm:iter_convergence}, we prove last-iterate convergence of $\RRM$ under the weaker assumptions \ref{B2}--\ref{B1}. Since $\RR$ is a special case of $\RRM$, this allows us to strengthen the convergence guarantees derived in \cite{li2023convergence}. Furthermore, in light of \ref{B1}, we may summarize the implications of \cref{thm:iter_convergence} as follows: we have either $x^k \to x^* \in \crit(f)$ or $\|x^k\| \to \infty$ (no other options can occur).
\end{remark}

\begin{remark} Under the additional geometric property \ref{B2}, the bound $\min_{k=1,\dots,T} \|\nabla f(x^k)\|^2 = o\big(({\sum_{k=1}^T \alpha_k})^{-1}\big)$, shown in \Cref{thm:global_convergence} (a), improves to $\min_{k=1,\dots,T} \|\nabla f(x^k)\|^2 = o\big(({\sum_{k=1}^T \alpha_k})^{-2}\big)$. As a result, under the setting of \cref{thm:iter_convergence} and if we consider polynomial step sizes of the form $\alpha_k \sim k^{-\gamma}$, we can infer
\[
    \min_{k=1,\dots,T} \|\nabla f(x^k)\|^2 \leq o({T^{-2(1-\gamma)}}), \quad T \to \infty,
    \]
provided that $\gamma \in (\frac12,1)$. (In this case, the conditions formulated in \eqref{thm:iter step size} are clearly satisfied). This improves the complexity bound $\min_{k=1,\ldots,T}\|\nabla f(x^k)\|^2 =\cO(1/T^{1-\gamma})$ from \cref{coro:complexity-polynomial}. Finally, we note that asymptotic complexity bounds similar to \Cref{thm:global_convergence} (a) have been shown in \cite{sebbouh2021almost,liu2022almost} for $\SGD$ and $\SGDM$. The results in \cite{sebbouh2021almost,liu2022almost} are based on the assumption $\sum_{k=1}^\infty \frac{\alpha_k}{\sum_{i=1}^{k-1}\alpha_i} = \infty$. Here, we can avoid this additional requirement by leveraging Kronecker's Lemma (cf.\ \Cref{lemma:kronecker}).
\end{remark}

\subsection{Proof of \texorpdfstring{\Cref{thm:iter_convergence}}{Theorem 6.3}}
The next lemma is the main local estimate used in the proof of \Cref{thm:iter_convergence}. It combines the descent estimate with the KL inequality and shows that, once the proxy iterate $z^k$ lies in a suitable (KL) neighborhood, the step length $\|z^{k+1}-z^k\|$ and the stationarity term $\alpha_k\|\nabla f(x^k)\|$ are controlled by the decrease of the ``desingularized'' term $\Psi_k=\varrho(\psi_k)$, up to summable errors. This is one of the key ingredients for proving finite length and hence last-iterate convergence. Since the proof is technical, we defer the detailed verification to \Cref{proof:lemma:kl-bound}.

\begin{lemma}[Local KL estimate]\label{lemma:kl-bound}
	Let the conditions stated in \cref{thm:global_convergence} hold and assume that \ref{B1}--\ref{B2} are satisfied. For arbitrary $x^* \in \cA$ and $\vartheta \in (0,1)$, let $\sC>0$, the neighborhood $U\subseteq\Rn$, and the desingularizing function $\varrho:[0,\eta)\to\R_+$ be given such that the conditions in \cref{lem:KL} hold. In addition, introducing $\psi_k:=f(z^{k}) - f(x^*) + u_k + \sH\alpha_{k}\|z^{k} - x^{k}\|^2$, suppose that $z^k \in  U$, $\sH\|z^k-x^k\|^2\leq 1$, and $0 < \psi_k \leq |f(z^k) - f(x^*)| + u_k + \sH\alpha_{k}\|z^{k} - x^{k}\|^2 < \eta$. Then, we have
\[
  \|z^{k+1}-z^{k}\| + \frac{m}{1-\beta} \alpha_k\|\nabla f(x^k)\|  
  \leq 40(\Psi_k - \Psi_{k+1}) +  \frac{\sC m}{1-\beta} (\alpha_k u_k^\vartheta + \alpha_k^{1+\vartheta}), 
\]
where $\Psi_k := \varrho(\psi_k)$ and $u_k :=\sG \sum_{i=k}^\infty\alpha_i^3$ and $\sG, \sH$ are defined in \eqref{eq:define-L} and \eqref{eq:def G and u}.
\end{lemma} 
 
\begin{proof}[Proof of \Cref{thm:iter_convergence}]
We first outline the flow of the proof. We begin by converting the approximate descent estimate in \cref{prop:approximate descent} into a descent recursion for the desingularized term $\Psi_k=\varrho(\psi_k)$. We then choose a sufficiently large index so that the proxy iterates $\{z^k\}_k$ enter a (KL) neighborhood of an accumulation point. The local KL estimate in \cref{lemma:kl-bound} is applied inductively to keep the iterates in this neighborhood and to establish the finite-length property of $\{z^k\}_k$, i.e., $\sum_{k=1}^\infty\|z^{k+1}-z^{k}\|<\infty$. Finally, this finite-length property yields the convergence of the proxy iterates $\{z^k\}_k$, and the convergence of the original iterates $\{x^k\}_k$ follows from $\|x^k-z^k\|\to0$.

\emph{Step 1: Conversion to a desingularized descent recursion.}
Let $x^*\in\cA$ be arbitrary. 
By \cref{thm:global_convergence}, we can infer that $x^*$ is a stationary point of $f$, i.e., $x^*\in\crit(f)$. Applying \Cref{thm:global_convergence} (c) and the continuity of $f$, it follows that $\lim_{k\to\infty} f(z^k) = \lim_{k\to\infty} f(x^k) = f^* = f(x^*)$.

Next, applying \cref{prop:approximate descent} (a) with $T = \infty$ (as in \eqref{eq:thm:approximate descent}), it again follows
\begin{equation*}
    \lya_{k+1} + u_{k+1} \leq \lya_k + u_k - \frac{1-\beta}{4m\alpha_k}\|z^{k+1} - z^{k}\|^2 - \frac{m\alpha_k}{4(1-\beta)}\Big[\frac{\|\nabla f(x^k)\|^2}{4} +\frac{\|\nabla f(z^k)\|^2}{5} \Big].
\end{equation*} 
Adding $\bar f-f^*$ on both sides and setting $\delta_k := \frac{m\alpha_k}{1-\beta}$ and $\psi_k := f(z^{k}) - f^* + u_{k} +  \sH \alpha_{k}\|z^{k} - x^{k}\|^2$, this bound can be expressed as
\begin{equation}
		\label{eq:lem:kl-key-4}
  \psi_{k+1} \leq \psi_k - \frac{1}{4\delta_k}\|z^{k+1} - z^{k}\|^2 - \frac{\delta_k}{20} \|\nabla f(z^k)\|^2 - \frac{\delta_k}{16}  \|\nabla f(x^k)\|^2.
\end{equation}
Since $f(z^{k}) \to f^*=f(x^*)$, $u_{k}\to0$ (by $\sum_{k=1}^{\infty}\alpha_k^3 < \infty$) and $\|z^{k} - x^{k}\|\to 0$ (cf.\ \Cref{prop:distance}), we conclude that $\psi_k\downarrow 0$ and $\psi_k \geq 0$ for all $k\geq 1$. Furthermore, without loss of generality, we may assume that $\psi_k > 0$ for all $k$. (This can be achieved by rescaling $u_k$ or $\sG$---if necessary).

\emph{Step 2: Entering the KL neighborhood.}
Based on \eqref{thm:iter step size}, we now choose $\vartheta = \xi$. Due to \ref{A2}, this leads to ${\sum}_{k=1}^{\infty} \alpha_k^{1+\vartheta}<\infty$ and
\[ {\sum}_{k=1}^{\infty} \alpha_k u_k^\vartheta = \sG^{\vartheta} \,{\sum}_{k=1}^{\infty} \alpha_k \Big({\sum}_{i=k}^\infty \alpha_i^3 \Big)^\vartheta \leq \sG^{\vartheta} \,{\sum}_{k=1}^{\infty} \alpha_k^{1+\vartheta} \cdot \Big({\sum}_{i=1}^\infty \alpha_i^2 \Big)^\vartheta < \infty. \]
Thus, we have $\sum_{i=k}^{\infty} \alpha_i^{1+\vartheta} \to 0$ and $\sum_{i=k}^{\infty}\alpha_i u_i^{\vartheta} \to 0$, $k \to \infty$. Using $x^*\in\cA$ and $\|x^k-z^k\| \to 0 $, there is a subsequence $\{\ell_k\}_k$ such that $\lim_{k\to\infty}z^{\ell_k} = x^*$. Hence, given $\rho>0$ with $\mathcal{B}(x^*,\rho) := \{x: \|x-x^*\| < \rho\} \subseteq U$, there is $k_\circ \geq 1$ such that
\begin{equation}
	\label{eq:thm-kl-2}
	\|z^{k_\circ} - x^*\| + 40 \Psi_{k_\circ} + \frac{\sC m}{1-\beta}{\sum}_{i=k_\circ}^{\infty} (\alpha_i  u_i^\vartheta + \alpha_i^{1+\vartheta})  < \rho
\end{equation}
and for all $k \geq k_\circ$, it holds that
\begin{equation}
	\label{eq:thm-kl-1}
	\sH\|z^k-x^k\|^2 \leq 1 \quad \text{and} \quad 0 < \psi_k \leq |f(z^k) - f(x^*)| + u_k + \sH \alpha_k \|z^k - x^k\|^2 < \eta.
\end{equation}

\emph{Step 3: Main induction.}
The key step of the proof is to show that the following statements are true for all $k \geq k_\circ$:
	\begin{enumerate}[label=(\alph*),topsep=0.5ex,itemsep=1ex,partopsep=0ex]
	\item We have $\sH\|z^k-x^k\|^2 \leq 1$, $0 < \psi_k \leq |f(z^k) - f(x^*)| + u_k + \sH\alpha_{k}\|z^{k} - x^{k}\|^2 < \eta $, and $z^k \in \mathcal{B}(x^*,\rho)$.
	\item $\sum_{i=k_\circ}^{k}\|z^{i+1} - z^{i}\| + \frac{m}{1-\beta} \alpha_i\|\nabla f(x^i)\|  \leq 40 [\Psi_{k_\circ}-\Psi_{k+1}] + \frac{\sC m}{1-\beta}\sum_{i=k_\circ}^{k} \alpha_i (u_i^\vartheta + \alpha_i^\vartheta)$. 
	\end{enumerate}
	We prove these claims by induction. Clearly, (a) \& (b) hold for $k=k_\circ$ according to \eqref{eq:thm-kl-2}, \eqref{eq:thm-kl-1} and \cref{lemma:kl-bound}. Let (a) \& (b) be valid for $k=k_\circ,\dots,t-1$ and consider $k=t$. By \eqref{eq:thm-kl-1}, we have $0 < \psi_k \leq |f(z^t) - f(x^*)| + u_t + \sH\alpha_{t}\|z^{t} - x^{t}\|^2 < \eta$ and $\sH\|z^{t}-x^{t}\|^2 \leq 1$. We now show $z^{t}\in \mathcal{B}(x^*,\rho)$. Using the triangle inequality and claim (b) with $k=t-1$, we obtain 
    \begin{equation*}
        \begin{aligned}
            \|z^{t} - x^*\| &\leq \|z^{k_\circ} - x^*\| +  {\sum}_{i= k_\circ}^{t-1}\|z^{i+1} - z^{i}\| \\&\leq \|z^{k_\circ} - x^*\| + 40 [\Psi_{k_\circ} - \Psi_{t}] + \frac{\sC m}{1-\beta}{\sum}_{i=k_\circ}^{t-1} \, \alpha_i (u_i^\vartheta + \alpha_i^\vartheta) < \rho,
        \end{aligned}
    \end{equation*}
	where the last step follows from \eqref{eq:thm-kl-2} and $\Psi_{t} \geq 0$. This proves (a) for $k=t$. Hence, \cref{lemma:kl-bound} is applicable for $z^{t}$ and we have $\|z^{t+1}-z^{t}\| + \frac{m}{1-\beta}\alpha_{t}\|\nabla f(x^{t})\|\leq 40(\Psi_{t} - \Psi_{t+1}) + \frac{\sC m }{1-\beta}\alpha_{t}(u_{t}^\vartheta + \alpha_{t}^\vartheta)$.
	Combining this bound with statement (b) (for $k=t-1$), this yields
	 \[{\sum}_{i=k_\circ}^{t}\|z^{i+1} - z^{i}\| + \frac{m}{1-\beta} \alpha_i\|\nabla f(x^i)\| \leq 40 (\Psi_{k_\circ} - \Psi_{t+1}) + \frac{\sC m}{1-\beta}{\sum}_{i=k_\circ}^{t} \;\alpha_i (u_i^\vartheta + \alpha_i^\vartheta),	\]
	 and consequently (b) is true for $k=t$. Therefore, the statements (a) \& (b) are valid for all $k \geq k_\circ$. Taking $k\to \infty$ in (b) and using \eqref{eq:thm-kl-2}, we can then infer 
	 \begin{equation} \label{eq:pf:finite-length}
	 {\sum}_{k=k_\circ}^{\infty}\;\Big(\|z^{k+1} - z^{k}\| + \frac{m}{1-\beta} \alpha_k\|\nabla f(x^k)\|\Big) < \rho <\infty.
	 \end{equation}

\emph{Step 4: Last-iterate convergence.}
	 The finite-length property \eqref{eq:pf:finite-length} implies that $\{z^k\}_k$ is a Cauchy sequence. Since $\|x^k-z^k\| \to 0$, the original iterates $\{x^k\}_k$ converge to the same limit $x^* \in \crit(f)$. The rate for $\min_{k = 1,\dots,T} \|\nabla f(x^k)\|^2$ follows from $\sum_{k=1}^\infty \alpha_k\|\nabla f(x^k)\| < \infty$ by applying Kronecker's lemma (\Cref{lemma:kronecker}), as in the proof of \Cref{thm:global_convergence}.
\end{proof}

\section{Preliminary numerical experiments}\label{sec:preliminary-experiments}

In this section, we conduct preliminary numerical experiments to examine the effects of the step sizes, momentum parameters, batch sizes, and sampling schemes on the nonconvex binary classification problem
\begin{equation}
	\label{eq:binary-clas-prelim}
	\min_{x\in \Rn}~f(x) := \frac{1}{n}\,{\sum}_{i=1}^{n}\,[1-\tanh(b_i\cdot a_i^\top x)] + \frac{\mu}{2}\|x\|^2.
\end{equation}
Here, the pairs $(a_i,b_i)$, $i \in [n]$, correspond to the data samples and labels, and $\tanh:\R\to\R$ denotes the hyperbolic tangent function. The Lipschitz constant of $\nabla f$ can be computed as $\sL:=0.8\|A\|_2^2/n$. This nonconvex classification problem was previously considered in \cite{wang2017stochastic,milzarek2019stochastic}. For the $\ell_2$-regularization parameter, we follow the choice in \cite{mishchenko2020random} and set $\mu = \sL/\sqrt{n}$. We test four binary classification datasets: \texttt{gisette} ($n = 6{,}000$, $d = 5{,}000$), \texttt{rcv1} ($n=20{,}242$, $d=47{,}236$), \texttt{sido0} ($n=12{,}678$, $d=4{,}932$), and \texttt{news20} ($n=19{,}996$, $d=1{,}355{,}191$)\footnote{Datasets are available at \url{www.csie.ntu.edu.tw/~cjlin/libsvmtools/datasets} and \url{www.causality.inf.ethz.ch/challenge.php?page=datasets\#cont}}. Performance is measured using the relative training error $(f(x^k)-f^*) / \min\{1,f^*\}$, where $f^*$ denotes the smallest function value achieved by the deterministic gradient method being run with high accuracy and starting from $10$ random initial points. In \Cref{fig:prelim-momentum,fig:prelim-batch,fig:prelim-sampling}, we report results for the heavy-ball version of $\RRM$, i.e., we run $\RRM$ with $\lambda=0$; 
in the momentum-parameter experiment (\Cref{sec:num:mom}), $\beta$ varies over $[0,1)$, while in the experiments shown in \Cref{sec:num:batch,sec:num:sampling}, we set $\beta=0.9$. We did not observe substantive differences when using the Nesterov option $\lambda=\beta$, and therefore we omit those plots to avoid repetition. All experiments were performed in \texttt{MATLAB R2026a} on a MacBook Pro with Apple M1 Max.

\begin{figure}[t]
\centering
\setlength{\abovecaptionskip}{-3pt plus 3pt minus 0pt}
\setlength{\belowcaptionskip}{-10pt plus 3pt minus 0pt}
\begin{tikzpicture}
\node[right] at (3,2) {\includegraphics[width=9cm,trim=110 360 105 120,clip]{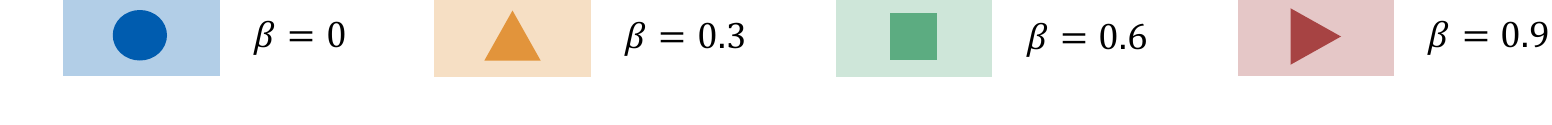}};
\node[right] at (0,0) {\includegraphics[width=3.7cm,trim=0 0 0 0,clip]{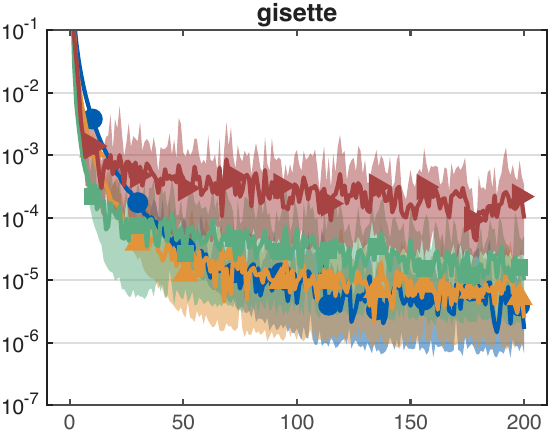}};
\node[right] at (4,0) {\includegraphics[width=3.7cm,trim=0 0 0 0,clip]{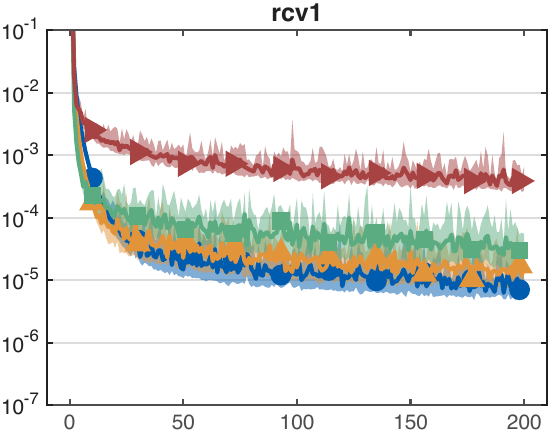}};
\node[right] at (8,0) {\includegraphics[width=3.7cm,trim=0 0 0 0,clip]{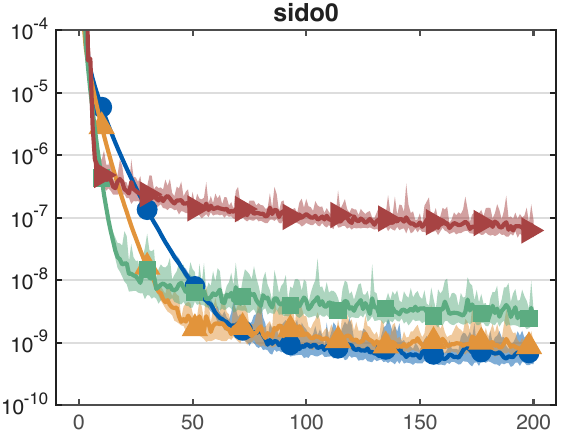}};
\node[right] at (12,0){\includegraphics[width=3.7cm,trim=0 0 0 0,clip]{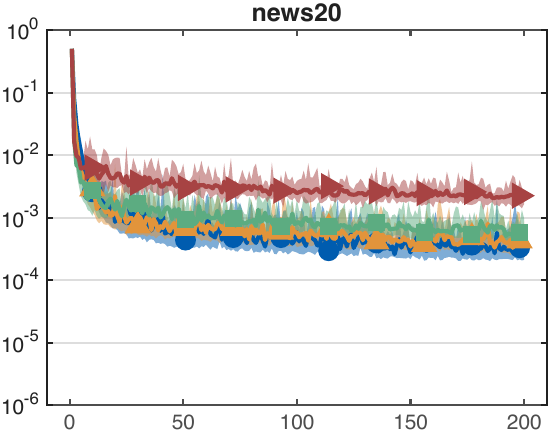}};
\node at (16.2,0) {\rotatebox{-90}{{\scriptsize $\gamma = 1/3$}}};
\node[right] at (0,-3.1) {\includegraphics[width=3.7cm,trim=0 0 0 0,clip]{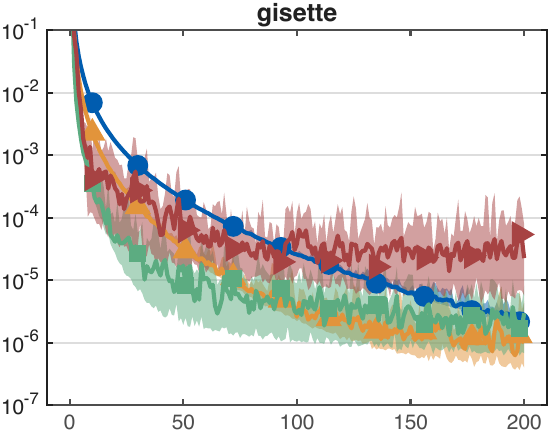}};
\node[right] at (4,-3.1) {\includegraphics[width=3.7cm,trim=0 0 0 0,clip]{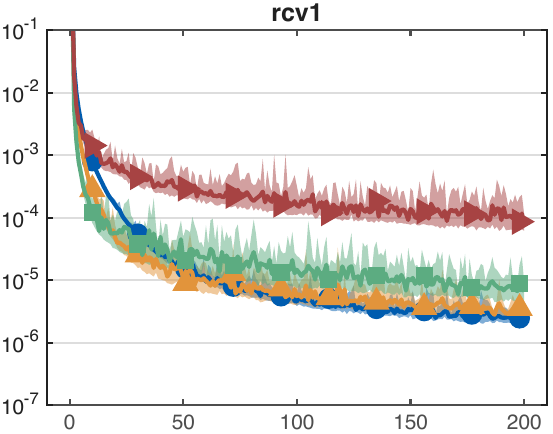}};
\node[right] at (8,-3.1) {\includegraphics[width=3.7cm,trim=0 0 0 0,clip]{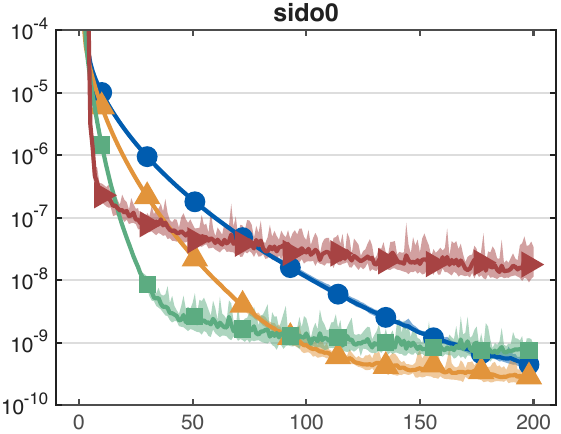}};
\node[right] at (12,-3.1){\includegraphics[width=3.7cm,trim=0 0 0 0,clip]{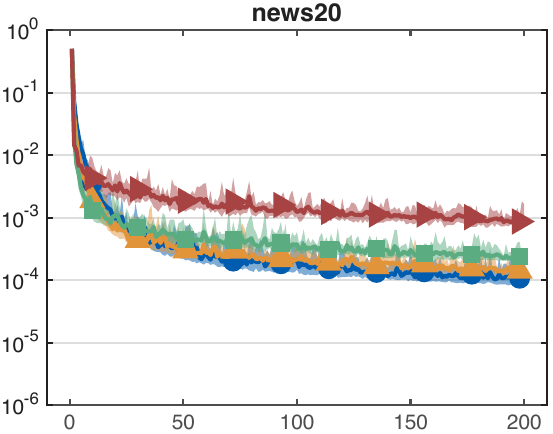}};
\node at (16.2,-3.1) {\rotatebox{-90}{{\scriptsize $\gamma = 1/2$}}};
\node[right] at (0,-6.2) {\includegraphics[width=3.7cm,trim=0 0 0 0,clip]{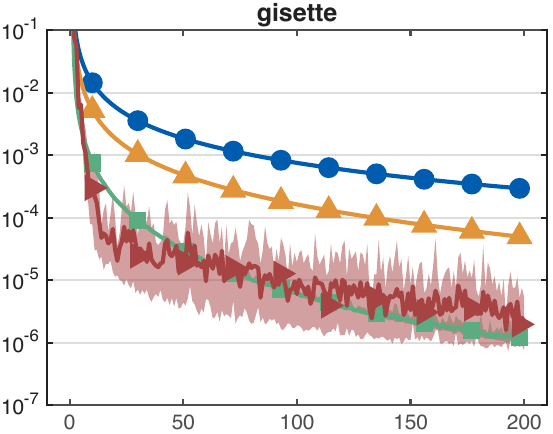}};
\node[right] at (4,-6.2) {\includegraphics[width=3.7cm,trim=0 0 0 0,clip]{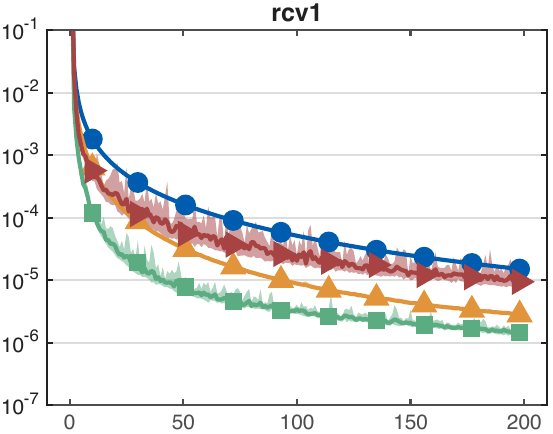}};
\node[right] at (8,-6.2) {\includegraphics[width=3.7cm,trim=0 0 0 0,clip]{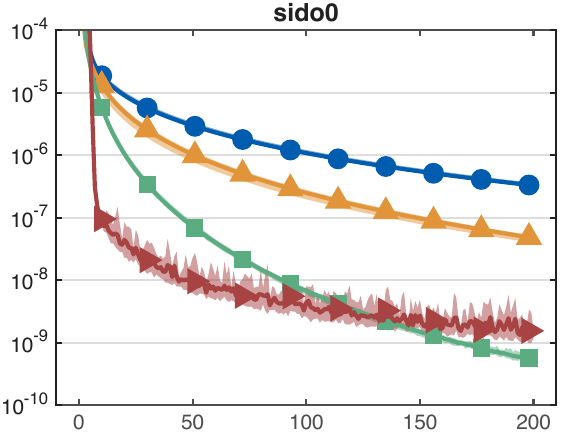}};
\node[right] at (12,-6.2){\includegraphics[width=3.7cm,trim=0 0 0 0,clip]{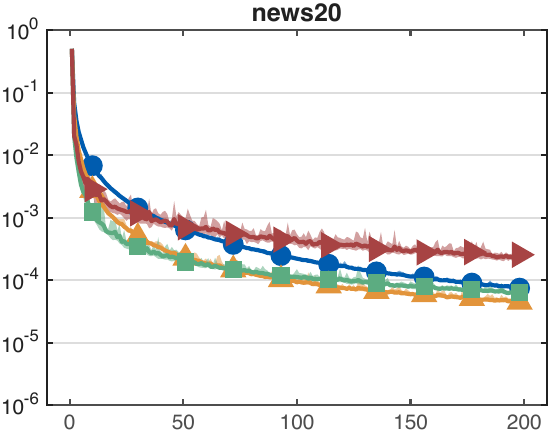}};
\node at (16.2,-6.2) {\rotatebox{-90}{{\scriptsize $\gamma = 3/4$}}};
\node[right] at (0,-9.3) {\includegraphics[width=3.7cm,trim=0 0 0 0,clip]{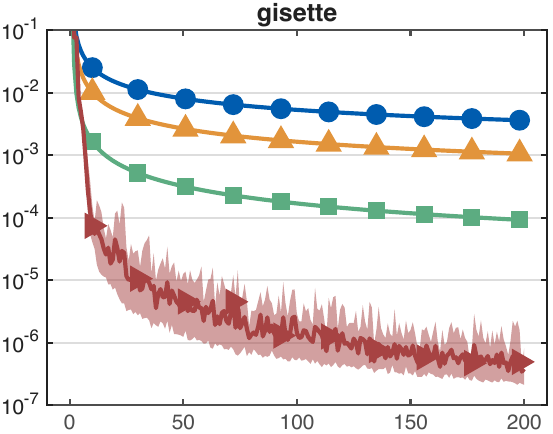}};
\node[right] at (4,-9.3) {\includegraphics[width=3.7cm,trim=0 0 0 0,clip]{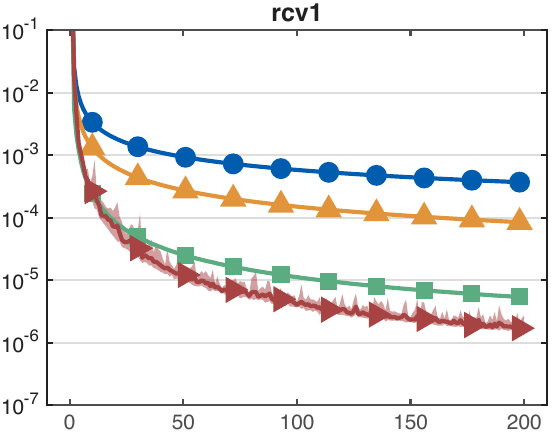}};
\node[right] at (8,-9.3) {\includegraphics[width=3.7cm,trim=0 0 0 0,clip]{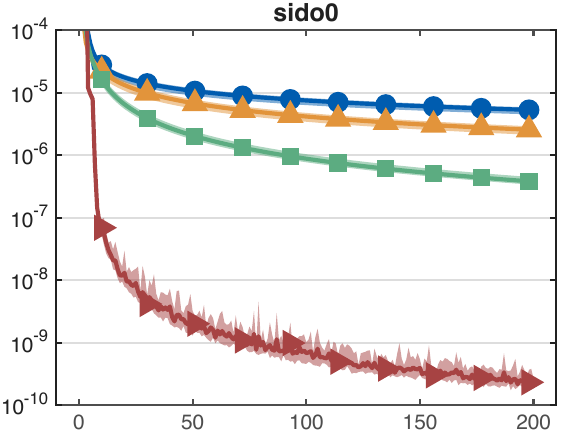}};
\node[right] at (12,-9.3){\includegraphics[width=3.7cm,trim=0 0 0 0,clip]{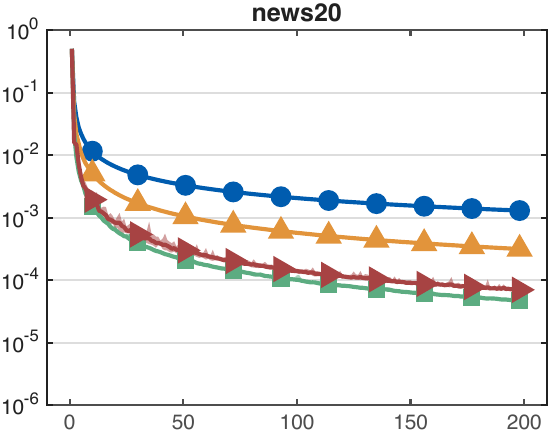}};
\node at (16.2,-9.3) {\rotatebox{-90}{{\scriptsize $\gamma = 1$}}};
\end{tikzpicture}
\caption{Performance of $\RRM$ for different step sizes $\alpha_k = 1/(Lk^\gamma)$, $\gamma \in \{\frac13,\frac12,\frac34,1\}$, $\beta\in[0,1)$, and $\lambda = 0$. The $x$-axis denotes the epoch $k$, and the $y$-axis denotes the relative training error $(f(x^k)-f^*)/\min\{1,f^*\}$. Averaged over $10$ independent runs.}
\label{fig:prelim-momentum}
\end{figure}

\subsection{Step sizes and effect of momentum terms} \label{sec:num:mom}
We first test $\RRM$ with fixed batch size $b=512$ using different polynomial steps sizes $\alpha_k = 1/(\sL k^\gamma)$, $\gamma \in \{\frac13,\frac12,\frac34,1\}$, and momentum parameters $\beta$. The preliminary results, depicted in \Cref{fig:prelim-momentum}, show that $\RRM$ performs effectively across a wide range of momentum parameters $\beta\in[0,1)$ and is not limited to small values of $\beta$. This observation is consistent with our theoretical results, which do not impose a small-momentum requirement and establish convergence for arbitrary $\beta\in[0,1)$. In the \texttt{gisette} and \texttt{sido0} datasets, setting $\beta=0.9$ and $\gamma = 1$ allows $\RRM$ to achieve a significantly lower objective value. More generally, larger choices of $\beta$ and $\gamma$ tend to lead to faster convergence in our experiments, suggesting that $\RRM$ can outperform the basic $\RR$ method.


\subsection{Effect of batch sizes} \label{sec:num:batch}
In \Cref{fig:prelim-batch}, we illustrate the performance of $\RRM$ for different batch sizes $b\in\{16,64,256,512\}$ and $\alpha_k = 1/(Lk)$, $\beta = 0.9$. To compare the batch sizes, we measure progress in terms of stochastic gradient steps rather than epochs. Each stochastic gradient step uses $b$ component-gradient evaluations, and each evaluation of $\nabla f_i$, $i\in[n]$, is counted as one gradient evaluation. The results show a clear and common trend: larger batch sizes lead to faster convergence. 

 
\subsection{Effect of sampling schemes} \label{sec:num:sampling}
Finally, in \Cref{fig:prelim-sampling}, we compare $\RRM$ with its deterministic version, the incremental gradient method with momentum ($\IGM$), the shuffle-once method with momentum ($\SOM$), and stochastic gradient descent with momentum ($\SGDM$). Here, $\SGDM$ uses a with-replacement sampling scheme, while $\RRM$ shuffles all the samples at each epoch. We consider the heavy-ball variants of those momentum methods, i.e., $\lambda=0$ and $\beta=0.9$. We use mini-batches of size $b=512$ and $\alpha_k = 1/(Lk)$ for all tested algorithms. The preliminary results indicate that $\RRM$ generally achieves faster convergence than the other methods. We also observe that $\SOM$, which reshuffles the data only once, performs similarly to $\IGM$, while $\SGDM$ often has the slowest progress. On the \texttt{news20} dataset, $\IGM$ converges slowly, likely due to an unfavorable data ordering; after shuffling the data, even just once, the performance improves.

\begin{figure}[t]
\centering
\setlength{\abovecaptionskip}{-3pt plus 3pt minus 0pt}
\setlength{\belowcaptionskip}{-10pt plus 3pt minus 0pt}
\begin{tikzpicture}
\node[right] at (3,2) {\includegraphics[width=9cm,trim=110 360 100 120,clip]{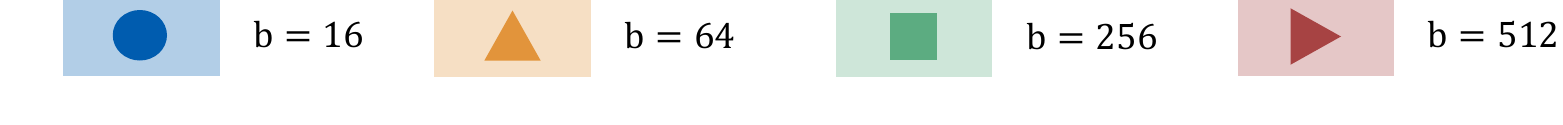}};
\node[right] at (0,0) {\includegraphics[width=3.7cm,trim=0 0 0 0,clip]{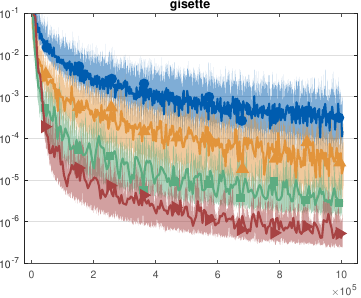}};
\node[right] at (4,0) {\includegraphics[width=3.7cm,trim=0 0 0 0,clip]{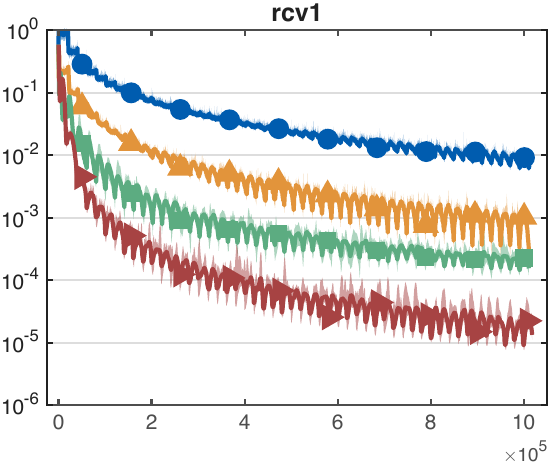}};
\node[right] at (8,0) {\includegraphics[width=3.7cm,trim=0 0 0 0,clip]{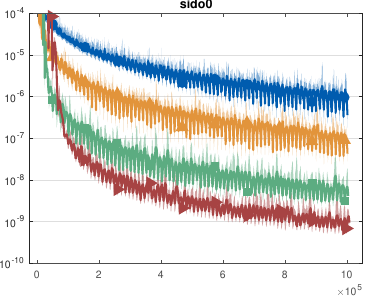}};
\node[right] at (12,0){\includegraphics[width=3.7cm,trim=0 0 0 0,clip]{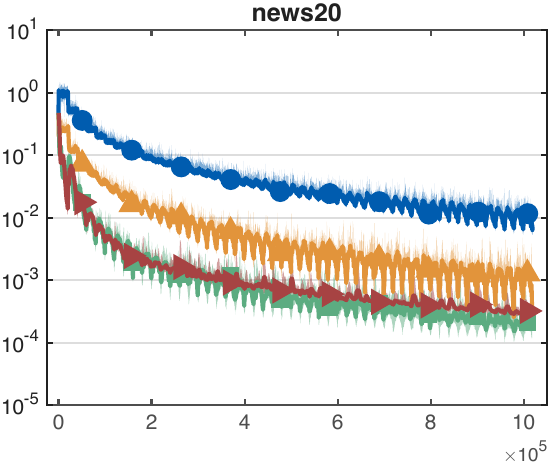}};
\end{tikzpicture} 
\caption{Performance of $\RRM$ with $\alpha_k = 1/(Lk)$, $\beta=0.9$, $\lambda = 0$ for different batch sizes $b$. The $x$-axis denotes number of gradient evaluations $\nabla f_i$, and the $y$-axis denotes the training error $(f(x^k)-f^*)/\min\{1,f^*\}$. Averaged over $10$ independent runs.}
\label{fig:prelim-batch}
\end{figure}

\begin{figure}[t]
\centering
\setlength{\abovecaptionskip}{-3pt plus 3pt minus 0pt}
\setlength{\belowcaptionskip}{-10pt plus 3pt minus 0pt}
\begin{tikzpicture}
\node[right] at (3,2) {\includegraphics[width=9cm,trim=110 360 105 120,clip]{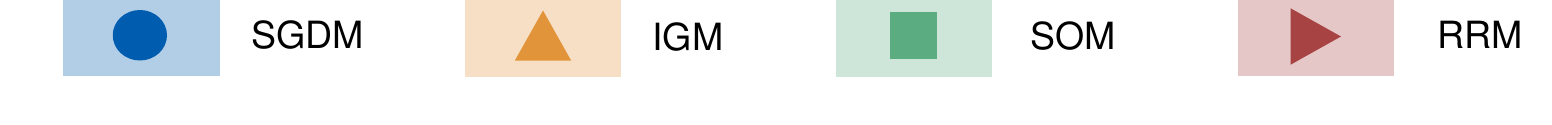}};
\node[right] at (0,0) {\includegraphics[width=3.7cm,trim=0 0 0 0,clip]{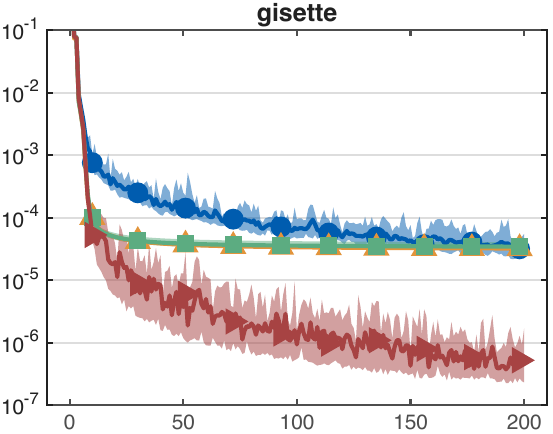}};
\node[right] at (4,0) {\includegraphics[width=3.7cm,trim=0 0 0 0,clip]{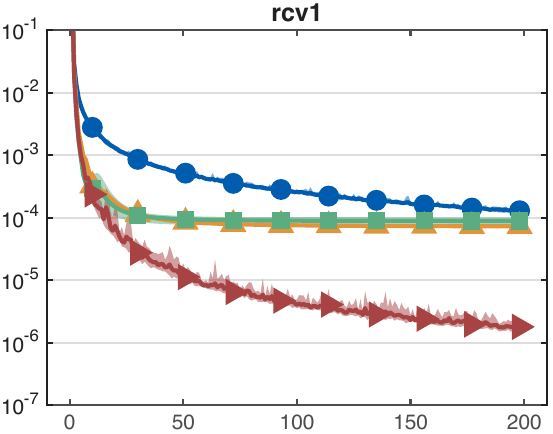}};
\node[right] at (8,0) {\includegraphics[width=3.7cm,trim=0 0 0 0,clip]{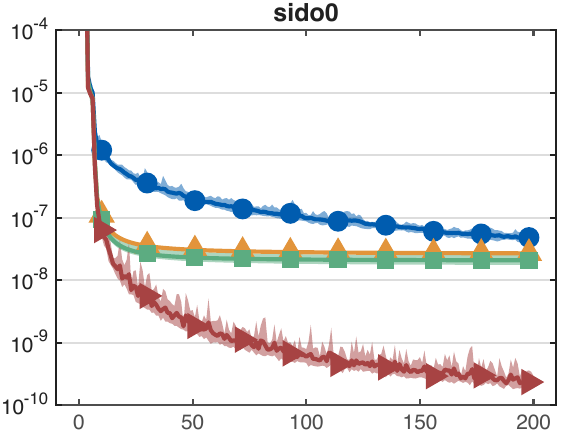}};
\node[right] at (12,0){\includegraphics[width=3.7cm,trim=0 0 0 0,clip]{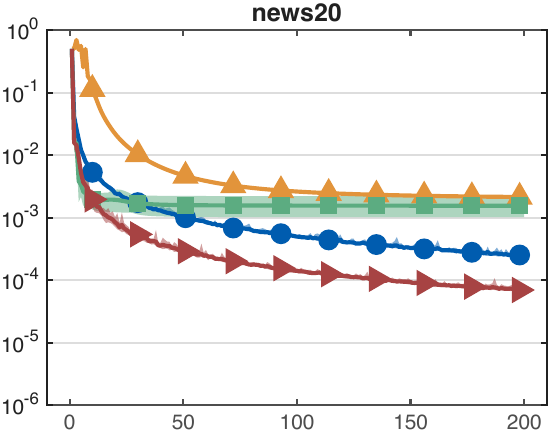}};
%
\end{tikzpicture}
\caption{Preliminary numerical results for solving \eqref{eq:binary-clas-prelim} with different stochastic algorithms. The $x$-axis denotes the epoch $k$, and the $y$-axis denotes the relative training error $(f(x^k)-f^*)/\min\{1,f^*\}$. Averaged over $10$ independent runs.}
\label{fig:prelim-sampling}
\end{figure}
\par


     \section{Conclusion and future directions}
This paper presents a comprehensive convergence analysis of the random reshuffling method with momentum ($\RRM$) with a primary focus on the nonconvex setting. Our theoretical framework and convergence results apply to arbitrary permutation schemes and momentum parameters $\beta\in [0,1)$ and $\lambda \in [0, \frac{\beta}{1-\beta}]$. As a consequence, our findings extend naturally to the incremental gradient method with momentum ($\IGM$). Under standard Lipschitz smoothness assumptions, we establish the first iteration complexity bounds for $\RRM$ that can fully match the existing complexity results for the basic random reshuffling method ($\RR$). Furthermore, we provide asymptotic last-iterate convergence guarantees under mild assumptions on the step sizes and objective function. Our KL-based techniques allow us to circumvent typical a priori boundedness conditions and seem of broader interest for the analysis of other stochastic optimization methods and momentum algorithms.

There are several possible future directions that may advance this line of research. First, it would be interesting to explore saddle point avoidance properties of $\RRM$ (see, e.g., \cite{pemantle1990nonconvergence,lee2019first} for saddle point avoidance results of first-order and stochastic approximation methods). Together with our iterate convergence guarantees, this would allow ensuring convergence to local minimizers. In addition, an examination of the asymptotic normality properties of $\RRM$, cf.\ \cite{fabian1968asymptotic}, can provide further insight into the statistical features of the generated iterates. Another interesting direction is to investigate whether the Nesterov variant of $\RRM$ admits improved complexity guarantees in the convex or strongly convex setting. Such a result would likely require a more nuanced analysis of the objective function and Lyapunov sequence for adaptive momentum parameters.

As $\RRM$ represents a special case of Adam \cite{kingma2014adam}, we anticipate that our analytical techniques for $\RRM$ may facilitate the derivation of more refined complexity results for Adam-type optimizers. In fact, most of the analyses of Adam require the stochastic gradients to be \emph{unbiased estimators}, \cite{kingma2014adam,reddi2019convergence,defossez2020simple}. However, practical implementations of Adam-type methods again use \emph{without-replacement} (shuffling-based) sampling schemes which induce \emph{bias}. The few existing works that account for such sampling strategies typically establish complexity bounds of the form $\mathcal{O}(1/\sqrt{T} + D)$ under the growth condition $\sum_{i=1}^n \|\nabla f_i(x)\|^2 \leq C\|\nabla f(x)\|^2 + D$ for some $C,D\geq 0$, \cite{shi2021rmsprop,zhang2022adam,wang2024provable}. Despite these remarkable advances and to the best of our knowledge, the current results seemingly cannot achieve the prototypical $\mathcal{O}(1/T^{2/3})$-bounds of basic shuffling methods.

\appendix
\crefalias{section}{appendix}
\renewcommand\theHsection{\Alph{section}}          
\renewcommand\theHsubsection{\Alph{section}.\arabic{subsection}}
\section{Preparatory tools}

In the following, we compute the variance of sampling a weighted collection of vectors from a finite set of vectors without replacement. In our analysis, the weights will frequently correspond to certain powers of the momentum parameter $\beta$ and hence, \cref{lemma:sampling} will play a key role in our theoretical derivations.

\begin{lemma}[Weighted sampling] \label{lemma:sampling}
Let $X_1,\dots,X_n \in \R^d$ and $w \in \R_+^t$, $t \in [n]$, be given and let $\bar X := \frac{1}{n} \sum_{i=1}^n X_i$ and $\sigma^2 := \frac{1}{n}\sum_{i=1}^n \|X_i - \bar X\|^2$ denote the associated average and population variance. Let $X_{\pi_1},\dots,X_{\pi_t}$ be sampled uniformly without replacement from $\{X_1,\dots, X_n\}$. Then, it holds that
\[
   \Exp\Big[\Big\| {\sum}_{i=1}^t w_i X_{\pi_i} - \Big({\sum}_{i=1}^t w_i \Big) \cdot \bar X \Big\|^2 \Big] \leq \|w\|^2\sigma^2.
\]
\end{lemma}

\begin{proof} As shown in \cite[Lemma 1]{mishchenko2020random}, we have $\Exp[\iprod{X_{\pi_i}-\bar X}{X_{\pi_j}-\bar X}] = -\frac{\sigma^2}{n-1}$ for all $i$, $j$ with $i \neq j$ and $\Exp[\|X_{\pi_i}-\bar X\|^2] = \sigma^2$. This yields
%
\begin{equation*}
\begin{aligned}
    &\hspace{-1ex}\Exp\left[\left\| {\sum}_{i=1}^t w_i X_{\pi_i} - \left({\sum}_{i=1}^t w_i \right) \cdot \bar X \right\|^2 \right] & \\ &  = \Exp\left[\left\| {\sum}_{i=1}^t w_i [X_{\pi_i} - \bar X] \right\|^2 \right]  = {\sum}_{i=1}^t {\sum}_{j=1}^t w_iw_j \Exp[\iprod{X_{\pi_i}-\bar X}{X_{\pi_j}-\bar X}] \\ &  = {\sum}_{i=1}^t w_i^2 \Exp[\|X_{\pi_i}-\bar X\|^2] - \left({\sum}_{i, j = 1, i \neq j}^t w_iw_j\right) \frac{\sigma^2}{n-1} 
\leq \|w\|^2 \sigma^2,
\end{aligned}
\end{equation*}
where the last inequality follows by dropping the nonpositive term.
\end{proof}  

%
Next, we present a weighted estimate used to handle the unequal gradient weights in $\RRM$.

\begin{lemma} \label{lemma:jensen} Let $u_1,\dots, u_t \in \Rn$ and $\lambda_1, \dots, \lambda_t \in \R_+$ be given vectors and scalars for some $t \in \N$. Then, it holds that 
\begin{equation*} \Big\|{\sum}_{i=1}^{t} \lambda_i u_i \Big\|^2 \leq {\sum}_{i=1}^t \lambda_i \cdot {\sum}_{i=1}^t \lambda_i \|u_i\|^2.\end{equation*}
\end{lemma}

\begin{proof}

If $\sum_{i=1}^t\lambda_i=0$, the claim is trivial. Otherwise, Jensen's inequality applied to $x \mapsto \|x\|^2$ gives
  \[
  \left\|\frac{\sum_{i=1}^t\lambda_i u_i}{\sum_{i=1}^t\lambda_i}\right\|^2 \leq  \frac{{\sum}_{i=1}^t \lambda_i\|u_i\|^2}{\sum_{i=1}^t\lambda_i}.
  \]
Multiplying both sides by $\left(\sum_{i=1}^t\lambda_i\right)^2$ yields the result.
\end{proof}

The following technical result, known as Kronecker's lemma, will be instrumental in our asymptotic complexity analysis, cf.\ \cite[Chapter 4, Section 3, Lemma 2]{shiryaev1996probability}.

\begin{lemma}[Kronecker's lemma] \label{lemma:kronecker} Let $\{r_k\}_k$ $\subset \R$ be given with $|\sum_{k=1}^\infty r_k | < \infty$ and let $\{s_k\}_k$ $\subset \R_{++}$ be a non-decreasing sequence with $s_k \to \infty$. Then, we have $s_k^{-1} {\sum}_{i=1}^k r_i s_i \to 0$.
\end{lemma}
 
\section{Key lemmas}\label{sec:key-lemmas}
In the following, we will use the term $\Exp_k[\cdot] := \Exp[\cdot | \cF_k]$ to denote the conditional expectation with respect to the $\sigma$-sub-algebra  $\mathcal F_k := \sigma\big(x^1,\{\pi^\ell,y_i^\ell:1\leq \ell \leq k-1,\;0\leq i\leq m+1\}\big)$.
In particular, $x^k$, $\tilde x^k$, and $z^k$ are $\mathcal F_k$-measurable, while the new permutation $\pi^k$ is sampled after conditioning on $\mathcal F_k$.

We now present several important results and key estimates that will be used to establish an approximate descent property for $\RRM$. We first derive bounds involving the stochastic gradient steps $d_i^k = \cG^k_i(\nes_{i}^{k})$, $i = 1,\dots, m$.

\begin{lemma}\label{lemma:weighted sum grad}
Let $\{x^k\}_k$, $\{\nes_i^k\}_{k,i\in[m]}$, and $\{d_i^k\}_{k,i\in[m]}$ be generated by $\RRM$ with mini-batch size $b$, $\beta\in[0,1)$, $\lambda \in \R$, $\{\alpha_k\}_k \subset \R_{++}$, and let $\{z^k\}_k$ be given as in \eqref{eq:def-z}. Let condition \ref{A1} hold. Then, the following statements are valid:  
\begin{enumerate}[label=\textup{(\alph*)},topsep=0pt,itemsep=0ex,partopsep=0ex,leftmargin=5ex]
    \item For all $i\in[m]$, $k\geq 1$, $w = (w_1,\dots,w_i)^\top \in \R_+^i$, and recalling $n = mb$, it holds that:
\[
    \Big\|{\sum}_{t=1}^i w_t d_t^k \Big\|^2 \leq 3 \|w\|_1\|w\|_\infty \Big[ \sL^2 {\sum}_{t=1}^m \|\nes_t^k-z^k\|^2 + m \sigma_k^2 \Big] + 3 \|w\|_1^2 \|\nabla f(z^k)\|^2.
\]
  \item In addition, if \ref{S1} is satisfied, then we have: 
\[
    \Exp_k\Big[\Big\|{\sum}_{t=1}^i w_t d^k_t \Big\|^2\Big] \leq 3 \|w\|_1\|w\|_\infty \Big[ \sL^2 {\sum}_{t=1}^m \Exp_k[\|\nes_t^k-z^k\|^2] + b^{-1} \sigma_k^2 \Big] + 3 \|w\|_1^2 \|\nabla f(z^k)\|^2.
  \]
\end{enumerate}
Here, the variance term $\sigma_k^2$ is given by $\sigma_k^2:=\frac1n{\sum}_{t=1}^n \|\nabla f_{t}(z^k) - \nabla f(z^k)\|^2$. 
\end{lemma}
 
\begin{proof}
Using the triangle inequality and the Lipschitz continuity of $\cG^k_t$, we obtain
\begin{equation*}
    \begin{aligned}
        &\hspace{5mm}\Big\|{\sum}_{t=1}^i w_t d^k_t \Big\|\\ & \leq \Big\|{\sum}_{t=1}^i  w_t  [\cG^k_t(\nes_{t}^{k}) - \cG^k_t(z^k)]\Big\| + \Big\|{\sum}_{t=1}^i w_t [\cG^k_t(z^k) - \nabla f(z^k)]\Big\| +\|w\|_1\|\nabla f(z^k)\| \\
& \leq \sL {\sum}_{t=1}^i w_t \|\hat y_{t}^{k} - z^k\| + \Big\|{\sum}_{t=1}^i w_t [\cG^k_t(z^k) - \nabla f(z^k)]\Big\| + \|w\|_1\|\nabla f(z^k)\|. 
    \end{aligned}
\end{equation*}
Taking squares on both sides, using $(\sum_{t=1}^3 b_t)^2 \leq 3 \sum_{t=1}^3 b_t^2$, and the definition of $\cG_t^k$, we have

\begin{equation*}
\begin{aligned}
    \frac13 \Big\|{\sum}_{t=1}^i w_t d^k_t \Big\|^2 & \leq \sL^2 \Big[{\sum}_{t=1}^i w_t\|\nes_t^k - z^k\|\Big]^2 + \frac{1}{b^2}\Big\|{\sum}_{t=1}^{ib} \tilde w_t [\nabla f_{\pi_t^k}(z^k) - \nabla f(z^k)]\Big\|^2 \\ & \hspace{4ex} +\|w\|_1^2\|\nabla f(z^k)\|^2,
\end{aligned}
\end{equation*}
where $\tilde w_t := w_\ell$ for $t = (\ell-1) b +1, \dots,\ell b$.
Invoking \cref{lemma:jensen}, it holds that $[{\sum}_{t=1}^i w_t\|\nes_t^k - z^k\|]^2 \leq \|w\|_1 \sum_{t=1}^i w_t \|\nes_t^k - z^k\|^2 \leq \|w\|_1 \|w\|_\infty \sum_{t=1}^i \|\nes_t^k - z^k\|^2 $ and
\begin{equation*}
\begin{aligned} \Big\|{\sum}_{t=1}^{ib} \tilde w_t [\nabla f_{\pi_t^k}(z^k) - \nabla f(z^k)]\Big\|^2 & \leq b \|w\|_1 {\sum}_{t=1}^{ib} \tilde w_t \|\nabla f_{t}(z^k) - \nabla f(z^k)\|^2 \\ & \leq n b \|w\|_1\|w\|_\infty \sigma_k^2 = m b^2 \|w\|_1\|w\|_\infty \sigma_k^2. 
\end{aligned}
\end{equation*}
Moreover, if condition \ref{S1} is satisfied, then \cref{lemma:sampling} is applicable and taking conditional expectation, we can infer $\Exp_k[\|{\sum}_{t=1}^{ib} \tilde w_t [\nabla f_{\pi_t^k}(z^k) - \nabla f(z^k)]\|^2 ] \leq b\|w\|^2\sigma_k^2 \leq b\|w\|_1\|w\|_\infty\sigma_k^2$. This completes the proof of \cref{lemma:weighted sum grad}. 
\end{proof}

Next, we provide an explicit update formula for the iterates $y_{i}^k$, $i=2,\dots,m+1$.

\begin{lemma}[Update rule]\label{lemma:update}
Let $\{x^k\}_k$, $\{\tilde x^k\}_{k}$, $\{y_i^k\}_{k,i\in[m+1]}$, and $\{d_i^k\}_{k,i\in[m]}$ be generated by $\RRM$ with mini-batch size $b$, $\beta\in[0,1)$, $\lambda \in \R$, and $\{\alpha_k\}_k \subset \R_{++}$. Then, it holds that
    \[y_{i+1}^k - x^k = - \alpha_k {\sum}_{t=1}^i \frac{1-\beta^{i-t+1}}{1-\beta} \cdot d^k_t + \frac{\beta(1-\beta^i)}{1-\beta}\cdot (x^k-\tilde x^k), \quad i \in [m]. \]
\end{lemma}

\begin{proof} If $\beta=0$, then $y_{i+1}^k-x^k=-\alpha_k\sum_{t=1}^{i}d_t^k$. It remains to consider the case $\beta>0$. Summing the iterative update of $\RRM$, it follows 
\begin{equation*}
\begin{aligned}
    y_{i+1}^k - x^k &= - \alpha_k {\sum}_{j=1}^i d^k_j + \beta (y_i^k -\tilde x^k) = - \alpha_k {\sum}_{j=1}^i d^k_j + \beta (y_i^k - x^k) +  \beta(x^k - \tilde x^k)
\end{aligned}
\end{equation*}
for all $i \in [m]$. Dividing both sides by $\beta^{i+1}$ and defining $v_i^k:=\alpha_k {\sum}_{t=1}^i d^k_t$, we further obtain
\begin{equation*} 
    \frac{y_{i+1}^k - x^k}{\beta^{i+1}} = \frac{y_{i}^k - x^k}{\beta^{i}} - \frac{v_i^k}{\beta^{i+1}} + \frac{x^k - \tilde x^k}{\beta^{i}} = \dots= - \sum_{j=1}^{i} \; \frac{v_j^k}{\beta^{j+1}} +  \sum_{j=1}^{i} \;\frac{x^k - \tilde x^k}{\beta^{j}},
\end{equation*}
which yields $y_{i+1}^k - x^k = - {\sum}_{j=1}^i \beta^{i-j} v_j^k + \beta{\sum}_{j=1}^{i}\beta^{i-j}(x^k - \tilde x^k)$ (after multiplying both sides of the previous equation with $\beta^{i+1}$). Hence, using $\sum_{j=1}^{i} \beta^{i-j} = \sum_{t=0}^{i-1} \beta^t = \frac{1-\beta^i}{1-\beta}$, it holds that
\begin{equation*}
\begin{aligned}
y_{i+1}^k - x^k    &= - \alpha_k {\sum}_{j=1}^i \beta^{i-j}{\sum}_{t=1}^j \;d^k_t + \frac{\beta(1-\beta^i)}{1-\beta}\cdot (x^k-\tilde x^k)\\
    &= - \alpha_k {\sum}_{t=1}^i \frac{1-\beta^{i-t+1}}{1-\beta} \cdot d^k_t + \frac{\beta(1-\beta^i)}{1-\beta}\cdot (x^k-\tilde x^k), 
\end{aligned}
\end{equation*}
which finishes the proof.
\end{proof}

\begin{lemma}\label{lemma:sum y-z}
Let $\{x^k\}_k$, $\{\hat y_i^k\}_{k,i\in[m]}$ be generated by $\RRM$ with mini-batch size $b$, $\beta\in[0,1)$, $\lambda \in [0,\frac{\beta}{1-\beta}]$, $\alpha_k \in (0,\frac{1-\beta}{\sqrt{8}\sL m}]$, and let $\{z^k\}_k$ be set as in \eqref{eq:def-z}. Define
\[
\chi := \chi(\beta,\lambda) :=
\begin{cases}
1-\frac{1-\beta}{\beta}\lambda, & \text{if}\;\;\beta\in(0,1)\;\;\text{and}\;\;\lambda\in[0,\frac{\beta}{1-\beta}],\\
0, & \text{if}\;\;\beta=0\;\;\text{and}\;\;\lambda=0.
\end{cases}
\]
If \ref{A1} holds, then:
\begin{enumerate}[label=\textup{(\alph*)},topsep=1ex,itemsep=0ex,partopsep=0ex, leftmargin = 5ex]
\item For all $k\geq 1$, we have
\[ {\sum}_{t=1}^{m}\|\hat y_{t}^k - z^k\|^2  \leq \frac{5m}{4}\Big[ \chi \|z^k-x^k\|^2 + \frac{m^2\alpha_k^2}{(1-\beta)^2} \|\nabla f(z^k)\|^2 + \frac{3\sL m^2\alpha_k^2}{(1-\beta)^2}[f(z^k) - \bar f]\Big]. \] 
\item In addition, if \ref{S1} is valid, it holds that
\[ {\sum}_{t=1}^{m} \Exp_k[\|\nes_{t}^k - z^k\|^2]  \leq \frac{5m}{4}\Big[\chi\|z^k-x^k\|^2 + \frac{m^2\alpha_k^2}{(1-\beta)^2} \|\nabla f(z^k)\|^2 + \frac{3\sL m\alpha_k^2}{b(1-\beta)^2}[f(z^k) - \bar f]\Big].\] 
\end{enumerate}
\end{lemma}

\begin{proof}  Applying \cref{lemma:update} and using $\frac{\beta}{1-\beta}(x^k-\tilde x^k) = z^k-x^k$, we have
\begin{equation}\label{eq:app:nes-z-identity}
\begin{aligned} \nes^k_{i+1} - z^k &= (1+\lambda) [y_{i+1}^k-x^k-(z^k-x^k)] - \lambda [y_i^k-x^k-(z^k-x^k)] \\ &= (1+\lambda)\Big[-\frac{\alpha_k}{1-\beta} {\sum}_{t=1}^i (1-\beta^{i-t+1})d_t^k - \beta^i(z^k-x^k) \Big] \\ & \hspace{4ex} + \lambda \Big[\frac{\alpha_k}{1-\beta} {\sum}_{t=1}^{i-1} (1-\beta^{i-t})d_t^k + \beta^{i-1}(z^k-x^k) \Big] \\ & = -\alpha_k {\sum}_{t=1}^i \Big[\frac{1-\beta^{i-t+1}}{1-\beta}+\lambda\beta^{i-t}\Big] \cdot d_t^k - \beta^{i-1}[\beta-(1-\beta)\lambda](z^k-x^k) \end{aligned}
\end{equation}
for $i = 1,\dots,m-1$. Setting $w_t := \frac{1-\beta^{i-t+1}}{1-\beta}+\lambda\beta^{i-t}$, $t \in [i]$, and using $\|u_1+u_2\|^2 \leq (1+\rho)\|u_1\|^2 + (1+\rho^{-1})\|u_2\|^2$, $u_1, u_2 \in \Rn$, $\rho > 0$, it follows
\begin{equation} \label{eq:app:esti-01} \|\nes^k_{i+1} - z^k\|^2 \leq (1+\rho)\beta^{2(i-1)}|\beta-(1-\beta)\lambda|^2 \|z^k-x^k\|^2 + (1+\rho^{-1}){\alpha_k^2} \Big\|{\sum}_{t=1}^i w_td_t^k\Big\|^2. \end{equation}
Since $\lambda \geq 0$, we have $w_t \geq 0$, and using $\lambda \leq \frac{\beta}{1-\beta}$, we obtain 
\[ w_t = \frac{1-\beta^{i-t}[\beta-(1-\beta)\lambda]}{1-\beta} \leq \frac{1-\beta^{i-1}[\beta-(1-\beta)\lambda]}{1-\beta} \quad \forall~t \in [i]. \] 
By the definition of $\chi$, this implies $\|w\|_\infty \leq \frac{1-\beta^i\chi}{1-\beta}$. If $\chi=0$, then $\beta-(1-\beta)\lambda=0$, and \eqref{eq:app:nes-z-identity} reduces to $\nes^k_{i+1}-z^k=-\alpha_k\sum_{t=1}^{i}w_td_t^k$. If $\chi>0$, choosing $\rho = \frac{1-\beta^i\chi}{\beta^i\chi}$ gives $1+\rho = \frac{1}{\beta^i\chi}$ and $1+\rho^{-1} = \frac{1}{1-\beta^i\chi}$. In both cases, applying \cref{lemma:weighted sum grad} (a) with $\|w\|_1 \leq i\|w\|_\infty$ yields
\[\|\nes^k_{i+1} - z^k\|^2  \leq \beta^{i} \chi \|z^k-x^k\|^2 + \frac{3(1-\beta^i\chi)\alpha_k^2}{(1-\beta)^2} \Big[\sL^2i {\sum}_{t=1}^{m}\|\nes_t^k-z^k\|^2 + mi\sigma_k^2 + i^2 \|\nabla f(z^k)\|^2 \Big]. \]
Summing this estimate for $i=1,\ldots,m-1$ and using $\nes_1^k - z^k = \chi(x^k-z^k)$, $\beta < 1$, and $\chi \in [0,1]$, we obtain
\begin{equation*}
\begin{aligned}
   {\sum}_{t=1}^{m}\|\nes_{t}^k - z^k\|^2 & = {\sum}_{i=1}^{m-1}\|\nes_{i+1}^k - z^k\|^2 + \|\nes_{1}^k - z^k\|^2  \\ & \hspace{-9ex} \leq \chi m \|z^k-x^k\|^2 + \frac{3m^2\alpha_k^2 }{2(1-\beta)^2}\Big[\sL^2{\sum}_{t=1}^m\|\nes_t^k - z^k\|^2 + m\sigma_k^2 \Big]+  \frac{m^3\alpha_k^2}{(1-\beta)^2}\|\nabla f(z^k)\|^2 , 
\end{aligned} 
\end{equation*}
where the last line uses $\sum_{i=1}^{m-1}i \leq \frac{m^2}{2}$ and $\sum_{i=1}^{m-1} i^2 \leq \frac{m^3}{3}$. Noting $\frac{3\sL^2m^2\alpha_k^2}{2(1-\beta)^2} \leq \frac{3}{16} < \frac15$ and rearranging the former estimate, it holds that
\[
    {\sum}_{t=1}^{m}\|\nes_{t}^k - z^k\|^2  \leq \frac{5\chi m}{4}\|z^k-x^k\|^2 + \frac{5m^3\alpha_k^2}{4(1-\beta)^2}\|\nabla f(z^k)\|^2 + \frac{15m^3\alpha_k^2}{8(1-\beta)^2}\sigma_k^2.
\]
The conclusion follows from $\sigma_k^2 \leq 2\sL[f(z^k) - \bar f]$. In order to prove part (b), we take conditional expectation in \eqref{eq:app:esti-01} and invoke \cref{lemma:weighted sum grad} (b). Mimicking our earlier steps, this yields
\[ \Exp_k[\|\nes^k_{i+1} - z^k\|^2] \leq \chi \|z^k-x^k\|^2 + \frac{3\alpha_k^2}{(1-\beta)^2} \Big[\sL^2i \sum_{t=1}^{m}\Exp_k[\|\nes_t^k-z^k\|^2] + \frac{i\sigma_k^2}{b} + i^2 \|\nabla f(z^k)\|^2 \Big]. \]
We can now simply repeat the last steps and derivations to establish the bound in (b).
\end{proof}
 
Finally, in \Cref{lemma:sum z-x}, we present a recursive expression for the terms $\|z^k-x^k\|$, $k \in \N$.

\begin{lemma}\label{lemma:sum z-x}
We consider the same setting as in \cref{lemma:sum y-z} with $\alpha_k \leq \frac{(1-\beta)(1-\beta^m)}{4\sL m}$. 
\begin{enumerate}[label=\textup{(\alph*)},topsep=1ex,itemsep=0ex,partopsep=0ex, leftmargin = 5ex]
\item Setting $\eta := \frac{1+2\beta^m}{3} \in [\frac13,1)$, it holds that:
\[
\|z^{k+1} - x^{k+1}\|^2 \leq  \eta \|z^k  -x^k\|^2 + \frac{\beta^2m^2\alpha_k^2\{4\|\nabla f(z^k)\|^2 + 7\sL[f(z^k) - \bar f]\}}{(1-\beta)^2(1-\beta^m)} \quad \forall~k \geq 1.
\]
\item In addition, if \ref{S1} is valid, we have for all $k\geq 1$ that
\[
    \Exp_k[\|z^{k+1} - x^{k+1}\|^2] \leq  \eta \|z^k  -x^k\|^2 + \frac{\beta^2m^2\alpha_k^2\{4\|\nabla f(z^k)\|^2 + 7n^{-1}\sL[f(z^k) - \bar f]\}}{(1-\beta)^2(1-\beta^m)}.
\]
\end{enumerate}
\end{lemma}

\begin{proof} If $\beta=0$, then $\lambda=0$ and $z^k=x^k$ for all $k$. Hence, $\|z^{k+1}-x^{k+1}\|=\|z^k-x^k\|=0$. Therefore, we may now assume $\beta>0$. By the update scheme of $\RRM$, we have
\begin{equation*}
\begin{aligned}  x^{k+1} - \tilde x^{k+1} & = y_{m+1}^k - y_m^k = - \alpha_k d^k_m + \beta (y_{m}^k - y_{m-1}^k) \\  & = - \alpha_k d^k_m + \beta [- \alpha_k d^k_{m-1} + \beta (y_{m-1}^k - y_{m-2}^k)]  = \dots  \\  & = - \alpha_k {\sum}_{j=1}^{m} \beta^{m-j} d^k_j + \beta^{m}(y_1^k - y_0^k) =  - \alpha_k {\sum}_{j=1}^{m} \beta^{m-j}d^k_j + \beta^{m}(x^{k} - \tilde x^{k}).
\end{aligned} 
\end{equation*}
Thanks to $z^k-x^k = \frac{\beta}{1-\beta}(x^k - \tilde x^k)$, the previous relation implies $\|z^{k+1} - x^{k+1}\| \leq  \beta^{m}\|z^{k} - x^{k}\| + \frac{\beta \alpha_k}{1-\beta}\|{\sum}_{j=1}^{m} \beta^{m-j} d^k_j\|$. Squaring both sides, using $(u+v)^2 \leq (1+\rho)u^2 + (1+\rho^{-1})v^2$ with $\rho=\frac{1-\beta^m}{\beta^m}$, and applying \cref{lemma:weighted sum grad} (a) directly, this yields
\begin{equation*}
\begin{aligned}
&\hspace{-2ex}\|z^{k+1} - x^{k+1}\|^2 \leq \beta^m\|z^{k} - x^{k}\|^2 + \frac{\beta^2\alpha_k^2}{(1-\beta)^2(1-\beta^m)}\Big\|{\sum}_{j=1}^{m} \beta^{m-j} d^k_j\Big\|^2\\
& \leq \beta^m\|z^{k} - x^{k}\|^2 + \frac{3\beta^2m\alpha_k^2}{(1-\beta)^2(1-\beta^m)} \Big[ \sL^2 {\sum}_{j=1}^m \|\nes_j^k-z^k\|^2 + m \sigma_k^2 + m \|\nabla f(z^k)\|^2 \Big] \\
&\leq \Big[ \beta^m + \frac{15\sL^2 \chi\beta^2 m^2 \alpha_k^2}{4(1-\beta)^2(1-\beta^m)}\Big]\|z^k  -x^k\|^2 \\
&\hspace{4ex}+\frac{3 \beta^2m^2\alpha_k^2}{(1-\beta)^2(1-\beta^m)}\Big\{\Big[1+\frac{5\sL^2m^2 \alpha_k^2}{4(1-\beta)^2} \Big]\|\nabla f(z^k)\|^2 + \Big[2+\frac{15\sL^2m^2 \alpha_k^2}{4(1-\beta)^2} \Big]\sL[f(z^k) - \bar f] \Big\},
\end{aligned}
\end{equation*}
where the last inequality is by \cref{lemma:sum y-z} (a) and $\sigma_k^2 \leq 2\sL[f(z^k) - \bar f]$. Since $\alpha_k \leq \frac{(1-\beta)(1-\beta^m)}{4\sL m}$ and $\chi \in [0,1]$, we have $\frac{5\sL^2 m^2\alpha_k^2}{4(1-\beta)^2}\leq\frac{5}{64} < \frac{1}{12}$, $\frac{15\sL^2 m^2\alpha_k^2}{4(1-\beta)^2}\leq\frac{15}{64} < \frac13$, and
\[\beta^m + \frac{15\sL^2\chi\beta^2 m^2 \alpha_k^2}{4(1-\beta)^2(1-\beta^m)} \leq \beta^m + \frac{\chi\beta^2(1-\beta^m)}{3} \leq \frac{1 + 2\beta^m}{3}=:\eta. 
\]
Therefore, combining the former estimates, we can infer
\[
\|z^{k+1} - x^{k+1}\|^2 \leq  \eta \|z^k  -x^k\|^2 + \frac{\beta^2m^2\alpha_k^2\{4\|\nabla f(z^k)\|^2 + 7\sL[f(z^k) - \bar f]\}}{(1-\beta)^2(1-\beta^m)}.
\]
This bound can be improved under assumption \ref{S1}. In particular, taking the conditional expectation, applying \Cref{lemma:weighted sum grad} (b) and \Cref{lemma:sum y-z} (b), and repeating the previous steps, we can readily obtain the bound shown in (b).
\end{proof}

\section{Proof of \texorpdfstring{\Cref{prop:approximate descent}}{Proposition 3.4}} \label{subsec:proof approximate descent}

\begin{proof} As shown in \eqref{eq:zk-works}, it holds that $(1-\beta) z^{k+1} = (1-\beta) z^k - \alpha_k {\sum}_{i=1}^m d^k_i$.
%
%
Applying the $\sL$-smoothness of $f$ and this relation, we obtain the following:
\begin{equation*}
\begin{aligned}
         f(z^{k+1}) &\leq f(z^k) - \langle \nabla f(z^k) , z^k - z^{k+1}  \rangle + \frac{\sL}{2}\|z^{k+1} - z^{k}\|^2\\
    &=  f(z^k) - \frac{m\alpha_k}{1-\beta} \Big\langle \nabla f(z^k) , \underbracket[.75pt][4pt]{m^{-1}{\sum}_{i=1}^m\, d^k_i }_{=:\,g^k}\Big\rangle + \frac{\sL}{2}\|z^{k+1} - z^{k}\|^2.
\end{aligned} 
\end{equation*}
Using $ -\big\langle \nabla f(z^k) , g^k \big\rangle = -  \tfrac12(\|\nabla f(z^k)\|^2 - \|\nabla f(z^k) - g^k\|^2)  -  \tfrac12 \|g^k \|^2$ and $g^k= \frac{1-\beta}{m\alpha_k} (z^k - z^{k+1})$, this yields
\begin{equation}
\begin{aligned}
    f(z^{k+1}) - f(z^k) & \leq - \frac{m\alpha_k}{2(1-\beta)}[\|\nabla f(z^k)\|^2-\|\nabla f(z^k) - g^k\|^2] - \Big[\frac{1-\beta}{2m\alpha_k} - \frac{\sL}{2}\Big]\|z^{k+1} - z^{k}\|^2 \\&\leq  - \frac{m\alpha_k}{2(1-\beta)}[\|\nabla f(z^k)\|^2-\|\nabla f(z^k) - g^k\|^2] - \frac{1-\beta}{4m\alpha_k}\|z^{k+1} - z^{k}\|^2,\label{eq:prop descent -1}
\end{aligned} 
\end{equation}
where the last line is due to $\frac{\sL}{2} \leq \tfrac{1-\beta}{4m\alpha_k}$. Next, we bound $ \|\nabla f(z^k) - g^k\|^2$. Using $d_i^k = \cG_i^k(\nes_i^k)$, $\|\sum_{i=1}^m u_i\|^2 \leq m \sum_{i=1}^m \|u_i\|^2$, $u_i \in \Rn$, and the $\sL$-continuity of $\cG_i^k:\Rn\to\Rn$, it follows that
\[
\|\nabla f(z^k) - g^k\|^2 = \frac{1}{m^2} \big\|{\sum}_{i=1}^m[\cG_i^k(\nes_i^k) - \cG_i^k(z^k)] \big\|^2 
    \leq \frac{\sL^2}{m} {\sum}_{i=1}^m \|\nes_i^k -z^k\|^2.
\]
Combining this estimate with \eqref{eq:prop descent -1} and subtracting $\bar f$, we have
\begin{equation}
    \label{eq:prop descent 0}
    \begin{aligned}
f(z^{k+1}) - \bar f \leq f(z^k) -\bar f &+ \frac{\sL^2\alpha_k}{2(1-\beta)} {\sum}_{i=1}^m \|\nes_i^k -z^k\|^2 \\ &- \frac{m\alpha_k}{2(1-\beta)}\|\nabla f(z^k)\|^2 - \frac{1-\beta}{4m\alpha_k}\|z^{k+1} - z^{k}\|^2.
\end{aligned}
\end{equation}
We now prove part (a). Applying \cref{lemma:sum y-z} (a), we can rewrite \eqref{eq:prop descent 0} as follows:
\begin{equation}
    \label{eq:prop descent 1}
    \begin{aligned}
[f(z^{k+1}) - \bar f] 
&\leq \Big[1+\frac{15\sL^3 m^3\alpha_k^3}{8(1-\beta)^3}\Big][f(z^k) -\bar f] + \frac{5\sL^2\chi m \alpha_k}{8(1-\beta)}\|z^k-x^k\|^2\\ & \hspace{4ex} - \frac{m\alpha_k}{2(1-\beta)}\Big[1-\frac{5\sL^2m^2\alpha_k^2}{4(1-\beta)^2} \Big]\|\nabla f(z^k)\|^2 - \frac{1-\beta}{4m\alpha_k}\|z^{k+1} - z^{k}\|^2.
\end{aligned}
\end{equation}
Using $\alpha_k \leq \frac{1-\beta}{4\sL m}$ and $\|u+v\|^2 \geq \frac12\|u\|^2 - \|v\|^2$, $u,v\in\Rn$, we have $\frac{5\sL^2m^2\alpha_k^2}{4(1-\beta)^2} \leq \frac{5}{64} \leq \frac{1}{12}$ and 
\begin{equation*}
\begin{aligned}
\Big[1-\frac{5\sL^2m^2\alpha_k^2}{4(1-\beta)^2} \Big]\|\nabla f(z^k)\|^2 
&\geq \frac23\|\nabla f(z^k)\|^2 + \frac14\|\nabla f(z^k) - \nabla f(x^k) + \nabla f(x^k)\|^2 \\&\geq \frac23\|\nabla f(z^k)\|^2 + \frac{1}{8}\|\nabla f(x^k)\|^2 - \frac{\sL^2}{4}\|x^k-z^k\|^2.
\end{aligned}
\end{equation*}
Combining this bound with \eqref{eq:prop descent 1} and using $\chi \in [0,1]$, we can infer 
\begin{equation*}
\begin{aligned}
&\hspace{-2ex} [f(z^{k+1}) - \bar f] + \frac{m\alpha_k}{16(1-\beta)} \|\nabla f(x^k)\|^2 + \frac{1-\beta}{4m\alpha_k}\|z^{k+1} - z^{k}\|^2
\\ & \leq \Big[1+\frac{15\sL^3 m^3\alpha_k^3}{8(1-\beta)^3}\Big][f(z^k) -\bar f] + \frac{3 \sL^2 m \alpha_k }{4(1-\beta)}\|z^k-x^k\|^2 - \frac{m\alpha_k}{3(1-\beta)}\|\nabla f(z^k)\|^2.
\end{aligned}
\end{equation*}
Based on the above estimate, applying \cref{lemma:sum z-x} (a), and defining $\delta_k := \frac{m\alpha_k}{1-\beta}[\frac{1}{16}\|\nabla f(x^k)\|^2 + \frac13\|\nabla f(z^k)\|^2] + \frac{1-\beta}{4m\alpha_k}\|z^{k+1}-z^k\|^2$, we further obtain
\begin{equation*}
\begin{aligned}
\nonumber \lya_{k+1} & = [f(z^{k+1})-\bar f] + \frac{9\sL^2m\alpha_{k+1}}{8(1-\beta)(1-\beta^m)} \|z^{k+1} - x^{k+1}\|^2 \\ 
\nonumber & \leq \Big\{1+\frac{\sL^3 m^3\alpha_k^3}{8(1-\beta)^3}\Big[{15}+\frac{63\beta^2}{(1-\beta^m)^2}\Big]\Big\}[f(z^k) -\bar f] - \delta_k   \\
\nonumber & \hspace{4ex} + \frac{3\sL^2 m \alpha_k }{4(1-\beta)}\Big[1 + \frac{3\eta}{2(1-\beta^m)}\Big]\|z^k-x^k\|^2 + \frac{9\sL^2\beta^2m^3\alpha_k^3}{2(1-\beta)^3(1-\beta^m)^2} \|\nabla f(z^k)\|^2 \\
& \leq (1+\sD m^3\alpha_k^3)[f(z^k) -\bar f] + \frac{9\sL^2 m \alpha_k }{8(1-\beta)(1-\beta^m)}\|z^k-x^k\|^2 - \delta_k + \frac{9m\alpha_k}{32(1-\beta)}\|\nabla f(z^k)\|^2
\end{aligned}
\end{equation*}
where the second inequality is because $\alpha_{k+1} \leq \alpha_k$ (cf. \ref{A2}), $1 + \frac{3\eta}{2(1-\beta^m)} = \frac{2(1-\beta^m)+1+2\beta^m}{2(1-\beta^m)} = \frac{3}{2(1-\beta^m)}$, $\frac{9\sL^2\beta^2m^2\alpha_k^2}{2(1-\beta)^2(1-\beta^m)^2} \leq \frac{9}{32}$, and
\[
\frac{1}{8}\Big[15 + \frac{63\beta^2}{(1-\beta^m)^{2}}\Big] \leq \frac{10}{(1-\beta^m)^2},\quad \text{and} \quad \sD = \frac{10}{(1-\beta^m)^2}\Big( \frac{\sL}{1-\beta} \Big)^3.
\]
Hence, due to $\frac{9}{32} - \frac13 \leq -\frac{1}{20} < 0$, for all $k = 1,\dots,T$, it follows that
\[
\lya_{k+1} \leq (1 + \sD   m^3\alpha_k^3)  \lya_k \leq \lya_1 \cdot  {\prod}_{i=1}^T (1 + \sD  m^3\alpha_i^3) \leq [f(z^1)-\bar f] \cdot \exp\Big(\sD \;{\sum}_{i=1}^T  m^3\alpha_i^3\Big),
\]
where we applied $1+x \leq \exp(x)$ and $\lya_1 = f(z^1)-\bar f$. Thus, recalling $\Delta(t):= [f(z^1)-\bar f]\cdot \exp(\sD t)$, we can conclude
\begin{equation*}
\begin{aligned}
    \lya_{k+1}  & \leq \lya_k + \Delta\Big({\sum}_{i=1}^T  m^3\alpha_i^3\Big) \sD m^3\alpha_k^3 \\ & \hspace{4ex} - \frac{m\alpha_k}{16(1-\beta)}\|\nabla f(x^k)\|^2 - \frac{m\alpha_k}{20(1-\beta)} \|\nabla f(z^k)\|^2 - \frac{1-\beta}{4m\alpha_k}\|z^{k+1} - z^{k}\|^2.
\end{aligned}
\end{equation*}

The overall proof strategy for (b) is analogous to part (a). Let us first derive an upper bound for $\lya_{k+1}$ under conditional expectation. According to the definition of $\lya_k$ and using \eqref{eq:prop descent 0}, it holds that
\begin{equation*}
\begin{aligned}
    \Exp_k[\lya_{k+1}] &\leq [f(z^k) -\bar f] - \frac{m\alpha_k}{2(1-\beta)}\|\nabla f(z^k)\|^2  + \frac{\sL^2\alpha_k}{2(1-\beta)} {\sum}_{i=1}^m \; \Exp_k[\|\nes_i^k -z^k\|^2]   \\ & \hspace{4ex} + \frac{9\sL^2m\alpha_{k+1}}{8(1-\beta)(1-\beta^m)}  \, \Exp_k[\|x^{k+1} - z^{k+1}\|^2].
\end{aligned}
\end{equation*}
Invoking \Cref{lemma:sum y-z} (b) and \Cref{lemma:sum z-x} (b) and repeating the steps of part (a), we can infer
\begin{equation*}
\begin{aligned}
\Exp_k[\lya_{k+1}]\leq (1 + \sD b^{-1} m^2\alpha_k^3) \lya_k - \frac{m\alpha_k}{16(1-\beta)} \|\nabla f(x^k)\|^2 - \frac{m\alpha_k}{20(1-\beta)} \|\nabla f(z^k)\|^2.
\end{aligned}
\end{equation*}
(Essentially, by \Cref{lemma:sum y-z} (b) and \Cref{lemma:sum z-x} (b), the terms depending on ``$\alpha_k^3[f(z^k)-\bar f]$'' will be rescaled by the factor $n^{-1} = (mb)^{-1}$). Taking the total expectation, this yields 
\[
\Exp[\lya_{k+1}]\leq (1 + \sD b^{-1} m^2\alpha_k^3) \cdot \Exp[\lya_k] - \frac{m\alpha_k}{16(1-\beta)}\Exp[\|\nabla f(x^k)\|^2].
\]
Given this recursion, we can derive the desired result by mimicking the earlier steps.   
\end{proof}

\section{Proof of \texorpdfstring{\Cref{lemma:kl-bound}}{Lemma 6.7}}\label{proof:lemma:kl-bound}

\begin{proof} 
By assumption, we have $z^k \in U$ and $0 < \psi_k \leq |f(z^{k}) - f^*| + u_{k} +  \sH \alpha_{k}\|z^{k} - x^{k}\|^2< \eta$. Let $\varrho$ be the desingularizing function introduced in \Cref{lem:KL}. Since $\varrho'$ is non-increasing (this follows from the concavity of $\varrho$), we have
\begin{equation}\label{eq:lem:kl-key-3e}
\begin{aligned}
         [\varrho^\prime(\psi_k)]^{-1} & \leq [\varrho^\prime(|\psi_k|)]^{-1}
         \leq [\varrho^\prime(|f(z^{k}) - f^*| + u_{k} +  \sH \alpha_{k}\|z^{k} - x^{k}\|^2)]^{-1}\\
         &\leq
         [\varrho^\prime(|f(z^{k}) - f^*|)]^{-1}
         + \sC \big[ u_k^\vartheta + (\sH\|z^k-x^k\|^2\alpha_k)^\vartheta \big]\\
        &\leq \|\nabla f(z^k)\| + \sC (u_k^\vartheta +  \alpha_k^\vartheta),
\end{aligned}
\end{equation}
where the second line follows from the property \eqref{eq:desingularization function} and
$(u+v)^\vartheta \leq u^\vartheta + v^\vartheta$ for all $u,v\geq 0$ and $\vartheta\in[0,1)$, and the last line is due to the KL inequality \eqref{eq:mer-kl} and $\sH\|z^k-x^k\|^2\leq 1$. (Due to $u_k \neq 0$, this estimate is also valid in the case $f(z^k) = f^*$). Invoking \eqref{eq:lem:kl-key-4} and the concavity of $\varrho$ and recalling $\delta_k = \frac{m\alpha_k}{1-\beta}$, it follows  
\begin{equation}
\begin{aligned}
   \hspace{3ex} \varrho(\psi_{k}) - \varrho(\psi_{k+1})  & \geq \varrho^\prime(\psi_k) \cdot (\psi_k - \psi_{k+1}) \\ 
     & \geq \varrho^\prime(\psi_k) \cdot \Big(\frac{\|z^{k+1} - z^{k}\|^2}{4\delta_k} + \frac{\delta_k}{20} \|\nabla f(z^k)\|^2 + \frac{\delta_k}{16}  \|\nabla f(x^k)\|^2 \Big)  \\
    &\geq \frac{1}{40} \cdot \frac{10 \delta_k^{-1} \|z^{k+1}-z^{k}\|^2 +  2\delta_k\|\nabla f(z^k)\|^2 +  2.5\delta_k\|\nabla f(x^k)\|^2}{\|\nabla f(z^k)\| + \sC (u_k^\vartheta +  \alpha_k^\vartheta)} \\[1mm]
    & \geq  
    \frac{1}{40} \cdot \frac{(\|z^{k+1}-z^{k}\| +  \delta_k\|\nabla f(z^k)\| +  \delta_k\|\nabla f(x^k)\|)^2}{\delta_k\|\nabla f(z^k)\| + \sC \delta_k (u_k^\vartheta + \alpha_k^\vartheta) }\\
    &\geq \frac{1}{40} \cdot [\|z^{k+1}-z^{k}\| +  \delta_k\|\nabla f(x^k)\|-\sC \delta_k (u_k^\vartheta + \alpha_k^\vartheta)], \label{eq:est-1}
\end{aligned}
\end{equation}
where 
 the third line is due to \eqref{eq:lem:kl-key-3e}, the fourth line uses $10 a^2+2b^2+2.5c^2 \geq (a+b+c)^2$, and the last line is due to $(a+b+c)^2/(b+d) \geq a+c-d$ with $a=\|z^{k+1}-z^{k}\|$, $b=\delta_k\|\nabla f(z^k)\|$, $c=\delta_k\|\nabla f(x^k)\|$, and $d=\sC \delta_k (u_k^\vartheta + \alpha_k^\vartheta)$. The desired result follows from rearranging \eqref{eq:est-1}.
\end{proof}

\section{Complexity bounds (\texorpdfstring{\cref{table:super-nice}}{Table 2}): Details} \label{app:table}

In the following, we provide more details about the complexity results reported in \cref{table:super-nice}. 

${}^{\textcolor{blue}{\text{(a)}}}$ The complexity bound for $\RR$ shown in \cref{table:super-nice} is taken directly from \cite{mishchenko2020random}.

${}^{\textcolor{blue}{\text{(b)}}}$ According to \cite[Theorem 1]{liu2020improved}, the complexity of $\SGDM$ is given by
\[ {\min}_{k=1,\dots,T}\; \Exp[\|\nabla f(x^k)\|^2] = \mathcal O({(T\alpha)}^{-1} + \sL\sB^2\alpha), \]
where the (constant) step size $\alpha$ satisfies $\alpha \lesssim \frac{1-\beta}{\sL}$, and $\frac{1}{n} \sum_{i=1}^n \|\nabla f_i(x) - \nabla f(x)\|^2 \leq \sB^2$ holds (cf.\ \cite[Assumption 1]{liu2020improved}). Choosing $\alpha = \min\{\frac{1-\beta}{\sL},\frac{\varepsilon^2}{\sL\sB^2}\}$ and requiring $T \geq \frac{\sL}{\varepsilon^2}\max\{\frac{1}{1-\beta},\frac{\sB^2}{\varepsilon^2}\}$, we can infer $\min_{k=1,\dots,T} \Exp[\|\nabla f(x^k)\|^2] \leq \cO(\varepsilon^2)$.

${}^{\textcolor{blue}{\text{(c)}}}$ The complexity results in \cite{tran2021smg} for $\SMG$ hold for the variance condition $\frac{1}{n} \sum_{i=1}^n \|\nabla f_i(x) - \nabla f(x)\|^2 \leq \sA \|\nabla f(x)\|^2+\sB^2$ (cf.\ \cite[Assumption 1]{tran2021smg}). Applying \cite[Theorem 2]{tran2021smg} with the constant step size parameter $\eta_t \equiv \alpha \lesssim \frac{\sqrt{1-\beta}}{\sL\sqrt{\sA/n+1}}$, the complexity bound of $\SMG$ is
\[ \min_{k=1,\dots,T} \Exp[\|\nabla f(x^k)\|^2] = \mathcal O\Big(\frac{1}{(1-\beta)T\alpha} + \frac{\sL^2\sB^2\alpha^2}{n(1-\beta)}\Big). \]
Thus, setting $\alpha = \frac{\sqrt{1-\beta}}{\sL} \min\{\frac{1}{\sqrt{\sA/n+1}},\frac{\sqrt{n}\varepsilon}{\sB}\}$, we can conclude
\[ \frac{1}{(1-\beta)T\alpha} + \frac{\sL^2\sB^2\alpha^2}{n(1-\beta)} \leq \frac{1}{Tn} \frac{\sL\sqrt{n}}{(1-\beta)^{3/2}}\max\Big\{\sqrt{\sA+n},\frac{\sB}{\varepsilon}\Big\} + \varepsilon^2 = \cO(\varepsilon^2), \]
provided that $Tn \geq \frac{\sL\sqrt{n}}{(1-\beta)^{3/2}\varepsilon^2}\max\{\sqrt{\sA+n},\frac{\sB}{\varepsilon}\}$.

${}^{\textcolor{blue}{\text{(d)}}}$ In \cite{tran2021smg}, complexity of $\RRM$ is studied under the assumption $\|\nabla f_i(x)\| \leq \sG$ for all $x, i$ and if a fixed permutation $\pi^k \equiv \pi$ is used. Applying \cite[Theorem 3]{tran2021smg} with $\alpha_k \equiv \alpha \leq \frac{1}{\sL}$, it holds that
\[ \min_{k=1,\dots,T} \Exp[\|\nabla f(x^k)\|^2] = \mathcal O\Big(\frac{1}{(1-\beta^n)T\alpha} + \sL^2\sG^2\alpha^2 + \frac{\beta^n\sG^2}{1-\beta^n}\Big). \]
Selecting $\alpha = \frac{1}{\sL}\min\{1,\frac{\varepsilon}{\sG}\}$ and $\beta^n \leq \frac{\varepsilon^2}{\sG^2+\varepsilon^2} \lesssim \frac{\varepsilon^2}{\sG^2}$ and requiring $Tn \geq [1+\frac{\varepsilon^2}{\sG^2}]\frac{\sL{n}}{\varepsilon^2}\max\{1,\frac{\sG}{\varepsilon}\}$, we have
\[ \frac{1}{(1-\beta^n)T\alpha} + \sL^2\sG^2\alpha^2 + \frac{\beta^n\sG^2}{1-\beta^n} \leq \frac{\sL}{T} \Big[1+\frac{\varepsilon^2}{\sG^2}\Big] \max\Big\{1,\frac{\sG}{\varepsilon}\Big\} + 2\varepsilon^2 = \cO(\varepsilon^2). \]

${}^{\textcolor{blue}{\text{(e)}}}$ Setting $\alpha = \min\{\frac{1}{4}, [\frac{n}{(1-\beta^m)T}]^{1/3}, [\frac{n}{\sL}]^{1/2}\varepsilon\}$ in \Cref{coro:complexity-constant}, it holds that
\begin{equation*}
 \begin{aligned}
 \min_{k=1,\ldots,T} \Exp[\|\nabla f(x^k)\|^2] &\leq \Big[ \frac{1}{(1-\beta^m)\alpha T} + \frac{3 \alpha^2}{n} \Big] \cdot 16 \sL [f(x^1) - \bar f] \\ & \hspace{-16ex} \leq \Big[\frac{\sL}{Tn}\max\Big\{\frac{4n}{1-\beta^m},\frac{\sqrt{\sL n}}{(1-\beta^m)\varepsilon},\frac{T^{1/3}n^{2/3}}{(1-\beta^m)^{2/3}}\Big\} + 3\varepsilon^2\Big] \cdot 16[f(x^1)-\bar f] = \cO(\varepsilon^2), 
 \end{aligned}
 \end{equation*}
provided that 
$Tn \geq \frac{\sL\sqrt{n}}{(1-\beta^m)\varepsilon^2}\max\{\sqrt{n},\sqrt{\sL}\varepsilon^{-1}\}$.

\bibliographystyle{siam} 
\bibliography{references}

\end{document}